\documentclass[11pt]{amsart}
\usepackage{amsmath, amssymb, amscd, mathrsfs, slashed, enumerate, url}
\usepackage{marginnote}
\usepackage[all,cmtip]{xy}
\usepackage{multicol}
\usepackage{indentfirst}
\usepackage{latexsym}
\usepackage{bm}
\usepackage{graphicx}
\usepackage{subfigure}
\usepackage{float}
\usepackage{verbatim}
\usepackage{comment}
\usepackage[top=1in, bottom=1in, left=1.25in, right=1.25in]{geometry}
\usepackage{epsfig,dsfont,amsthm,amsfonts,amsbsy}
\usepackage[numbers]{natbib}         
\usepackage[colorlinks]{hyperref}

\newcommand{\note}[1]{\marginpar{\raggedright\if@twoside\ifodd\c@page\raggedleft\fi\fi\sf\scriptsize {#1}}}
\newcommand{\gpp}{\mathfrak{g}_P}

\newcommand{\oti}{\otimes}
\newcommand{\GP}{\mathcal{G}_P}

\newcommand{\MD}{\mathcal{D}}

\newcommand{\MM}{\mathcal{M}}

\newcommand{\MO}{\mathcal{O}}

\newcommand{\Vol}{\mathrm{Vol}}

\newtheorem{theorem}{Theorem}[section]

\newtheorem{corollary}[theorem]{Corollary}

\newtheorem{definition}[theorem]{Definition}
\newtheorem{example}[theorem]{Example}

\newtheorem{lemma}[theorem]{Lemma}

\newtheorem{proposition}[theorem]{Proposition}
\newtheorem*{remark}{Remark}

\newcommand{\MC}{\mathcal{C}}
\newcommand{\Tr}{\mathrm{Tr}}
\newcommand{\pab}{\bar{\partial}}
\newcommand{\st}{\star}
\newcommand{\we}{\wedge}
\newcommand{\pa}{\partial}
\newcommand{\Bp}{\partial_{\bz}}

\newcommand{\RP}{\mathbb R^+}
\newcommand{\EBE}{extended Bogomolny equations\;}
\newcommand{\ti}{\times}
\newcommand{\SLR}{SL(2,\mathbb{R})}
\newcommand{\Si}{\Sigma}
\newcommand{\bz}{\bar{z}}
\newcommand{\vp}{\varphi}
\newcommand{\MA}{\mathcal{A}}
\newcommand{\ME}{\mathcal{E}}
\newcommand{\da}{\dagger}
\newcommand{\al}{\alpha}

\newcommand{\na}{\nabla}
\newcommand{\ep}{\epsilon}
\newcommand{\hA}{\widehat{A}}

\newcommand{\hP}{\widehat{\Phi}}
\newcommand{\bpa}{\bar{\pa}}
\newcommand{\hu}{\hat{u}}
\newcommand{\yrz}{y\rightarrow 0}

\newcommand{\Hom}{\mathrm{Hom}}
\newcommand{\be}{\beta}
\newcommand{\hB}{\hat{B}}
\newcommand{\bb}{\bar{b}}

\newcommand{\calC}{\mathcal C}

\newcommand{\calL}{\mathcal L}

\newcommand{\calU}{\mathcal U}

\newcommand{\calX}{\mathcal X}
\newcommand{\del}{\partial}
\newcommand{\RR}{\mathbb R}
\newcommand{\CC}{\mathbb C}
\newcommand{\wh}{\widehat}
\newcommand{\Fusi}{\mathrm{Hit}}
\newcommand{\Fusic}{\mathrm{Hit}^c}
\newcommand{\NP}{\mathrm{NP}}
\newcommand{\KS}{\mathrm{KS}}
\newcommand{\EBEt}{\mathrm{EBE}}
\newcommand{\wt}{\widetilde}

\usepackage[utf8]{inputenc}
\usepackage[english]{babel}
\usepackage{fancyhdr}

\begin{document}
\title[The Extended Bogomolny Equations]{The Extended Bogomolny Equations and Generalized Nahm Pole Boundary Condition}
\author{Siqi He} 
\address{Department of Mathematics, California Institute of Technology\\Pasadena, CA, 91106}
\email{she@caltech.edu}
\author{Rafe Mazzeo}
\address{Department of Mathematics, Stanford University\\Stanford,CA 94305 USA}
\email{rmazzeo@stanford.edu}

\begin{abstract}
In this paper we develop a Kobayashi-Hitchin type correspondence between solutions of the extended Bogomolny equations on 
$\Sigma\times \RP$ 
with Nahm pole singularity at $\Sigma \times \{0\}$ and the Hitchin component of the stable $SL(2,\mathbb{R})$ Higgs bundle;
this verifies a conjecture of Gaiotto and Witten. We also develop a partial Kobayashi-Hitchin correspondence for solutions with a 
knot singularity in this program, corresponding to the non-Hitchin components in the moduli space of stable $SL(2,\mathbb{R})$ 
Higgs bundles. We also prove existence and uniqueness of solutions with knot singularities on $\mathbb{C}\ti\RP$.
\end{abstract}
\maketitle

\begin{section}{Introduction}
An intriguing proposal by Witten \cite{witten2011fivebranes} interprets the Jones polynomial and Khovanov homology of knots on a 
$3$-manifold $Y$ by counting solutions to certain gauge-theoretic equations, see \cite{KapustinWitten2006}, \cite{witten2011fivebranes}, 
\cite{Haydys2015Fukaya} for much more on this.  In this picture, the Jones polynomial for a knot $K \subset Y$ is realized by a count of solutions 
to the Kapustin-Witten equations on $Y\ti\RP$ satisfying a new type of singular boundary conditions. We refer \cite{gaiotto2012knot}, 
\cite{Witten2014LecturesJonesPolynomial}, \cite{Witten2016LecturesGaugeTheory} for a more detailed explanation, along with \cite{MazzeoWitten2013}, 
\cite{MazzeoWitten2017} and \cite{He2017} for the beginnings of the analytic theory for this program.  In the absence of a knot, the problem 
is still of interest and may lead to $3$-manifold invariants.  When $K = \emptyset$, the singular boundary conditions are called the Nahm pole
boundary conditions, while in the presence of a knot, they are called the generalized Nahm pole boundary conditions, or Nahm pole boundary
conditions with knot singularities.  For simplicity, we usually just refer to solutions with Nahm pole or with Nahm pole and knot singularities.

\begin{figure}[H]
\centering
\includegraphics[width=0.5\textwidth]{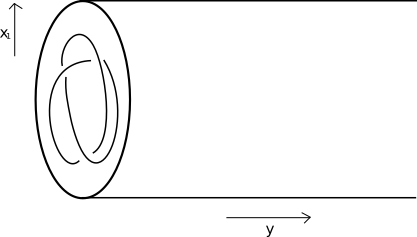}
\caption{A knot placed at the boundary of $Y\ti\RP$}
\label{KWpic1}
\end{figure}

There are two main sets of technical difficulties in this program. The first arises from the singular boundary conditions, which turn the problem into
one of nonstandard elliptic type. These are now understood, see \cite{MazzeoWitten2013}, \cite{MazzeoWitten2017}.  A more serious difficulty 
involves whether it is possible to prove compactness of the space of solutions to the Kapustin-Witten (KW) equations. An important first step 
was accomplished by Taubes in \cite{Taubes20133manifoldcompactness}, \cite{taubes2013compactness},
but at present there is no understanding about how the Nahm pole boundary conditions interact with these compactness issues. 

Gaiotto and Witten \cite{gaiotto2012knot} proposed the study of a more tractable aspect of this problem. 
Suppose that we stretch the $3$-manifold across a separating Riemann surface $\Si$ in a Heegard decomposition of $Y$ which meets the knot 
transversely. In the limit, $Y$ separates into two components $Y^\pm$ and zooming in on the transition region leads to a problem on
$\Si \ti \RR \ti \RP$ which is independent of the $\RR$ direction normal to the separating surface. We are thus led to study the dimensionally reduced 
problem, called the extended Bogomolny equations, on $\Sigma\ti\RP$ with the induced singular boundary condition.   

A further motivation for studying the moduli space of solutions of the \EBE on $\Sigma\ti\RP$ is provided by the Atiyah-Floer conjecture \cite{Atiyah1987AFloer}. 
In terms of a handlebody decomposition $Y^3= Y^+ \cup_{\Sigma} Y^-$, the Atiyah-Floer conjecture states that the instanton Floer homology of $Y$ 
can be recovered from Lagrangian Floer homology of two Lagrangians associated to the handlebodies in the moduli space $\MM(\Sigma)$
of flat $SU(2)$ connection of $\Sigma$.  These Lagrangians consist of the flat connections which extend into $Y^+$ or $Y^-$.  Another way to
view $\MM(\Si)$ is as the moduli space for the reduction of the anti-selfdual equations to $\Sigma$.  One then expects to use Lagrangian intersectional 
Floer theory to define invariants. We refer to \cite{FukayaDaemi2017atiyah}, \cite{ManolescuAbouzaid} for recent progress on this.

In any case, we are presented with the problem of studying the dimensionally reduced Kapustin-Witten equations on $\Sigma\ti\RP$ with generalized
Nahm pole boundary conditions. We describe these now; their derivation and further explicit computations appear in Section 2 below. 
Let $P$ be a principal $SU(2)$ bundle over $\Si$, pulled back to $\Si \ti \RP$, and $\gpp$ its adjoint bundle. The \EBE are the following set of equations
for a connection $A$ on $P$, and $\gpp$-valued $1$- and $0$-forms $\phi$ and $\phi_1$, respectively: 
\begin{equation}
\begin{split}
F_A-\phi\we\phi&=\st d_A\phi_1,\\
d_A\phi&=\st[\phi,\phi_1],\\
d^{\st}_A\phi&=0.
\label{EBE}
\end{split}
\end{equation} 
The knot corresponds in this setting to where the stretched knot crosses $\Si$, or in other words, to a set of marked points 
$\{p_1, \ldots, p_N\}$ on $\Si$, see Figure \ref{KWpic2}. 

\begin{figure}[H]
\centering
\includegraphics[width=0.4\textwidth]{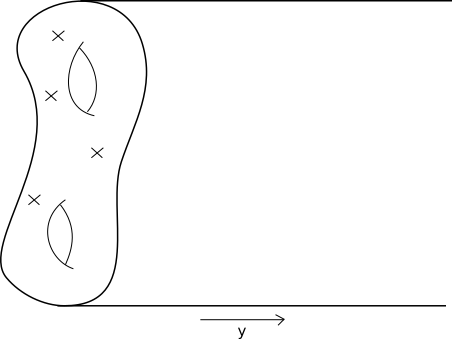}
\caption{$\Si \ti \RP$; the `knots' correspond to points on $\Si\ti\{0\}$}
\label{KWpic2}
\end{figure}

In the following we the standard linear coordinate $y$ on $\RP$.  Define $\MM^{\EBEt}_{\NP}$ and $\MM^{\EBEt}_{\KS}$ to be the moduli spaces 
of solutions to \eqref{EBE} which satisfy the Nahm pole, and generalized Nahm pole, boundary conditions at $y=0$, and which converge to an 
$SL(2,\mathbb{R})$ flat connection as $y \to \infty$. For the second of these spaces, we tacitly restrict to the subset of solutions which
are compatible with a $SL(2,\mathbb{R})$ structure, as explained more carefully in Section 3.  The subscripts NP and KS here stand for
`Nahm pole' and `knot singularity'. We also write $\MM$ for the moduli space of stable $SL(2,\mathbb{R})$ Higgs pairs
and recall that $\MM = \MM^{\Fusi} \sqcup \MM^{\Fusic}$, where the first term on the right is the Fuchsian, or Hitchin, component and 
$\MM^{\Fusic}$ the union of the other components. It is well-known that $\MM^{\Fusi}$ identified with a finite cover of the Techm\"uller space for $\Si$. 

In the spirit of Donaldson-Uhlenbeck-Yau \cite{donaldson1985anti},\cite{uhlenbeck1986existence}, Gaiotto and Witten \cite{gaiotto2012knot} 
define maps
\begin{equation}
\begin{split} I_{\NP}&:\MM^{\EBEt}_{\NP}\to \MM^{\Fusi},\\ I_{\KS}&:\MM^{\EBEt}_{\KS}\to \MM^{\Fusic}
\end{split}
\end{equation}
which we recall in Section 3. They conjecture that $I_{NP}$ is one-to-one.  We prove this here and also describe the map $I_{KS}$. 
Our main result is:
\begin{theorem}
(i) The map $I_{\NP}$ is bijection. Explicitly, to every element in the Hitchin component $\MM^{\Fusi}$, there exists a solution to \eqref{EBE} 
satisfying the Nahm pole boundary condition. If two solutions to \eqref{EBE} satisfying these boundary conditions map to the same element 
in $\MM^{\Fusi}$ under $I_{\NP}$, then they are $SU(2)$-gauge equivalent. 

(ii) The map $I_{\KS}$ is two-to-one: for every element in the $\MM^{\Fusic}$, there exist two solutions to \eqref{EBE} which satisfy generalized
Nahm pole boundary conditions with knot singularities and which are compatible with the $SL(2,\mathbb{R})$ structure as $y\to\infty$. 
Any solution to \eqref{EBE} satisfying these boundary and compatibility conditions is equal, up to $SU(2)$-gauge equivalence, with 
one of these two solutions.
\end{theorem}

We define in Section 3 what it means for solutions of \eqref{EBE} with knot singularities to be compatible with the $SL(2,\mathbb{R})$ structure 
as $y\to\infty$. This condition allows \eqref{EBE} to be reduced to a scalar equation.  There are almost surely solutions to \eqref{EBE}
which do not satisfy this condition. 

The expectation, explained in \cite{witten2011fivebranes}, is that the Jones polynomial should be recovered by counting solutions to the \EBE
on $\mathbb{R}^3\ti\RP$, with a knot singularity at some $K \subset \RR^3$. Thus, as a dimensionally reduced version of this problem,
we also consider these equations on $\mathbb{C}\ti\RP$: 
\begin{theorem}
Given any positive divisor $D=\sum n_ip_i$ on $\CC$, there exists a solution to \eqref{EBE} which has knot singularities of order $n_i$ at $p_i$.
This solution is unique to the scalar equation. 
\end{theorem}

\textbf{Acknowledgements.}  The first author wishes to thank Ciprian Manolescu, Qiongling Li and Victor Mikhaylov. 
The second author is grateful to Edward Witten for introducing him to this problem originally and for his many patient explanations. 
The second author has been supported by the NSF grant DMS-1608223.

\end{section}


\begin{section}{Preliminaries}
We begin by considering various ways in which the extended Bogomolny equations \eqref{EBE} may be interpreted. 
\begin{subsection}{$S^1$-Invariant Kapustin-Witten Equations}
Let $X$ be a smooth 4-manifold with boundary, $P$ an $SU(2)$ bundle over $X$ and $\gpp$ the adjoint bundle of $P$. 
If $\hA$ is a connection on $P$ and $\hP$ is a $\gpp$-valued one-form, then the Kapustin-Witten equations for
the pair $(\hA, \hP)$ are 
\begin{equation}
\begin{split}
F_{\hA}-\hP\we\hP+\st d_{\hA}\hP&=0,\\
d_{\hA}^{\st}\hP&=0.
\label{KW}
\end{split}
\end{equation}

Consider the special case where $X=S^1\ti Y$ is the product of a circle and a $3$-manifold, and where 
$(\hA, \hP)$ is an $S^1$ invariant solution to \eqref{KW}. We then set
\begin{equation}
\hA=A+A_1dx_1,\;\hP=\phi+\phi_1dx_1,
\end{equation}
where $A,\phi\in\Omega_Y^1(\gpp)$ and $A_1,\phi_1\in\Omega_Y^0(\gpp)$ are independent of $x_1 \in S^1$. Then \eqref{KW} becomes
\begin{equation}
\begin{split}
F_A-\phi\we\phi+\st d_A\phi_1+\st[A_1,\phi]&=0,\\
\st d_A\phi -[\phi_1,\phi]-d_A A_1&=0,\\
d_A^{\st}\phi-[A_1,\phi_1]&=0.
\label{GEBE}
\end{split}
\end{equation}

Denoting the quantities on the left of these three qualities by $\calX_1$, $\calX_2$ and $\calX_3$, respectively, we define the expressions
\begin{equation}
\begin{split}
I_0 &= \int_Y |\calX_1|^2 + |\calX_2|^2 + |\calX_3|^2 \\ 
I_1&=\int_Y|F_A-\phi\we\phi+\st d_A\phi_1|^2+|\st d_A\phi -[\phi_1,\phi]|^2+|d_A^{\st}\phi|^2,\\
I_2&=\int_{Y}|[A_1,\phi]|^2+|d_A A_1|^2+|[A_1,\phi_1]|^2,\\
\end{split}
\end{equation}
and also, if $Y$ is a 3-manifold with boundary, 
\[
I_3=-\int_{\pa Y}\Tr(d_A A_1\we\phi_1)-\int_{\pa Y}\Tr([A_1,\phi_1]\we \st \phi).
\]
After a straightforward calculation, assuming that all integrations are valid, we have
\begin{equation}
I_0=I_1+I_2+I_3.
\end{equation}
Since $I_0,I_1,I_2$ are all nonnegative, we deduce the 
\begin{proposition}
If $(A_1,\phi_1)$ satisfies a boundary condition which guarantees that $I_3=0$, and if $(A,\phi)$ is irreducible, then 
$A_1=0$ and \eqref{GEBE} reduces to the equations corresponding to $I_1=0$.
\label{prop1}
\end{proposition}
The case of principal interest in this paper is when $Y=\Si\ti\RP_y$ and $(\hA,\hP)$ satisfy the Nahm pole boundary conditions
at $y=0$ and converge as $y \to \infty$ to a flat $SL(2,\mathbb{R})$ connection.  The conditions of this proposition 
are then satisfied.  We recall the claim, see \cite[Page 36]{taubes2013compactness} as well as \cite[Corollary 4.7]{He2017}, that for
solutions satisfying these boundary conditions, the $dy$ component of $\phi$ vanishes. Results from \cite{MazzeoWitten2013} show
that as $y\searrow 0$, $A_1\sim y^2$ and $\phi_1\sim \frac{1}{y}$, hence $\st\phi=0$ at $y=0$. In addition, $A_1$ and 
$\phi_1$ both converge to $0$ as $y \to \infty$.  These facts together imply that $I_3$ vanishes at both $y=0$ and  $y = \infty$, 
so Proposition \ref{prop1} holds.

If an $S^1$-invariant solution satisfies the Nahm pole boundary condition at $y=0$ and converges to a flat 
$SL(2,\mathbb{R})$ connection as $y \to\infty$, then the pair $(A,\Phi)$ satisfies the so-called extended 
Bogomolny equations on $\Sigma\ti\RP$:
\begin{equation}
\begin{split}
F_A-\phi\we\phi&=\st d_A\phi_1,\\
d_A\phi&=\st[\phi,\phi_1],\\
d^{\st}_A\phi&=0.
\end{split}
\end{equation}
Here $A$ is a connection, $\phi\in\Omega^1(\gpp)$, $\phi_1\in\Omega^0(\gpp)$ and the $dy$ component of $\phi$ vanishes.

These equations reduce, when $\phi_1 = 0$, to the Hitchin equations, when $\phi = 0$, to the Bogomolny equations, 
and when $A=0$ and $\phi$ is independent of $\Si$, to the Nahm equations. Thus one expects that all known techniques 
for these special cases should be applicable to these hybrid equations as well.
\end{subsection}

\begin{subsection}{Hermitian Geometry}
Choose a holomorphic coordinate $z=x_2+ix_3$ on $\Sigma$ and let $y$ be the linear coordinate on $\RP$. 
In these coordinates, define $d_A=\nabla_2\, dx_2+\nabla_3\, dx_3+\nabla_y \, dy$ and $\phi=\phi_2\, dx_2+\phi_3 \, dx_3= 
\vp_z \, dz+\vp^{\da}_{\bz}\, d\bz$, where $\vp_z = \phi_2 - i \phi_3$; we also write $\vp=\vp_z dz$.  Using these, we can 
rewrite \eqref{EBE} in the ``three D's'' formalism: with $\MA_y=A_y-i\phi_1$, set
\begin{equation}
\begin{split}
\MD_1&=\na_2+i\na_3,\\
\MD_2&=\mathrm{ad}\, \vp = [\vp,\cdot ] ,\\
\MD_3&=\na_y-i\phi_1=\partial_y+\MA_y=\partial_y+A_y-i\phi_1.
\end{split}
\end{equation}
The adjoints of these operators are
\begin{equation}
\begin{split}
\MD_1^{\da}&=-\na_2+i\na_3,\\
\MD_2^{\da}&=-[\phi_2+i\phi_3,\cdot ],\\
\MD_3^{\da}&=-\na_y-i\phi_1.
\end{split}
\end{equation}
The \EBE can then be written in the alternate form 
\begin{equation}
[\MD_i,\;\MD_j]=0, \ i, j = 1, 2, 3, \ \mbox{and}\ \sum_{i=1}^3[\MD_i,\;\MD_i^\da]=0.
\label{algebraEBE}
\end{equation}
We write out the last of these, which is the most intricate. Noting that
\begin{equation}
\begin{split}
[\MD_1,\;\MD_1^\da]&=[\na_2+i\na_3,-\na_2+i\na_3]=2iF_{23},\\
[\MD_2,\;\MD_2^\da]&=-2i[\phi_2,\phi_3],\\
[\MD_3,\;\MD_3^\da]&=-2i\na_y\phi_1,
\end{split}
\end{equation}
we have
\[
\frac{1}{2i}\sum_{k=1}^3[\MD_i,\;\MD_i^\da]= F_{23}-[\phi_2,\phi_3]-\na_y\phi_1=0.
\]

As is standard for such equations, cf.\ \cite{witten2011fivebranes}, the smaller system $[\MD_i,\MD_j]=0$ is invariant under 
the complex ($SL(2;\mathbb{C})$-valued) gauge group $\GP^{\mathbb C}$, while the full system \eqref{algebraEBE} is invariant 
under the unitary gauge group, $\MD_i\to g^{-1}\MD_ig$, $g\in \GP$ and the final equation is a real moment map condition.
Following the spirit of Donaldson-Uhlenbeck-Yau \cite{donaldson1985anti},\cite{uhlenbeck1986existence},  we thus
expect that Hermitian geometric data from the $\GP^{\mathbb C}$-invariant equations play a role in solving the moment map equation.

Suppose that $E$ is a rank $2$ Hermitian bundle over $\Si \ti \RP$.  As we now explain, $\MD_1$ induces a holomorphic structure which
makes $E$ into a holomorphic bundle $\ME$; $\MD_2$ is then a $K_\Si$-valued endomorphism of $\ME$, while $\MD_3$ specifies 
a parallel transport in the $y$ direction. In terms of these, the equations $[\MD_i, \MD_j]=0$ have a nice geometric meaning.

Denote by $E_y:=E|_{\Si\ti \{y\}}$ the restriction of $E$ to each slice $\Sigma\ti \{y\}$. Since $\MD_1^2=0$ is always true for dimensional
reasons, the Newlander-Nirenberg theorem gives that $\MD_1$ induces a holomorphic structure on $E_y$ for each $y$, i.e.,  $\MD_1 = \pab$. 
A connection $A$ is compatible with this holomorphic structure if $A^{0,1}$ equals $\pab$.

Next, $[\MD_1,\MD_2]=0$ says that the endomorphism $\vp$ is holomorphic with respect to this structure, so 
$(E,\MD_1,\vp)$ is a Higgs pair over each slice. Finally, the equations $[\MD_2,\MD_3]=0$, $[\MD_1,\MD_3]=0$ 
show that this family of Higgs pairs is parallel in $y$, i.e., there is a specified identification of these objects at
different values of $y$.  

Following \cite{donaldson1985anti}, a data set for our problem consists of a rank two bundle $E$ over 
$\Sigma\ti\RP$ and a triplet of operators $\Theta = (\MD_1,\MD_2,\MD_3)$ on $\calC^{\infty}(E)$ satisfying 
\begin{itemize}
\item $\MD_1(fs) = \del_{\bz}f s + f \MD_1 s$, $\MD_3(fs) = (\del_{y}f) s + f \MD_3 s$ for 
$f\in \calC^{\infty}(\Sigma\ti\RP)$  and $s\in C^{\infty}(E)$; 
\item $\MD_2=[\vp, \cdot ]$ for some $\vp\in \Omega^1(\gpp)$; 
\item $[\MD_i,\;\MD_j]=0$ for all $i, j$. 
\end{itemize}

Given $(E,\Theta)$, a choice of Hermitian metric $H$ on $E$ determines Hermitian adjoints $\MD_i'$ of the operators $\MD_i$ by the
requirements that for any smooth functions $f$ and sections $s$: 
\begin{itemize}
\item $\MD_1'$ and $\MD_3'$ are derivations, i.e., $\MD_1' (fs) = (\del_{z} f) s + f \MD_1 s$, $\MD_3' (fs) = (\del_{y} f) s + f \MD_3 s$,
while $\MD_2(fs) = f \MD_2(s)$; 
\item $\pa_{\bz}H(s,s')=H(\MD_1 s,s')+H(s,\MD_1's'),\;\pa_{y}H(s,s')=H(\MD_3 s,s')+H(s,\MD_3's')$;
\item $H(\MD_2 s,s')+H(s,\MD_2^{'}s') = 0$
\end{itemize}


The moment map equation in \eqref{algebraEBE} can be regarded as an equation for the Hermitian metric $H$.  Indeed, setting
$\MD_y=\frac{1}{2}(\MD_3+\MD_3^{'})$, $\MD_z=\MD_1$ and $\MD_{\bz}=\MD_1^{'}$, we define a unitary connection $\MD_{A}$, and an
endomorphism-valued $1$-form $\phi$ and $0$-form $\phi_1$ on $(E,\Theta,H)$ by 
\begin{equation}
\begin{split}
\MD_A(s):&=\MD_1(s)dz+\MD_1^{'}(s)d\bz+\MD_y(s)dy,\\
[\phi,s]:&=[\MD_2,s]dz+[\MD_2^{'},s] d\bz,\\
\phi_1:&=\frac{i}{2}(\MD_3-\MD_3^{'}).
\end{split}
\label{relationship}
\end{equation}
We call $(A,\phi,\phi_1)$ a unitary triplet.  Note however that in an arbitrary trivialization of $E$, $(A,\phi,\phi_1)$ may not consist of unitary matrices. 
We recall a standard result \cite{atiyah1978geometry} which provides the link between connections in unitary and holomorphic frames. In the
following, and later, we refer to parallel holomorphic gauges. These are, as the moniker suggests, holomorphic gauges for each $E_y$ which
are parallel with respect to $\MD_3$. 
\begin{proposition}
With $(E,\Theta, H)$ as above, there is a unique triplet $(A,\phi,\phi_y)$ compatible with the unitary structure and 
with the structure defined by $\Theta$. In other words, in every unitary gauge, $A^{\st}=-A$, $\phi^{\st}=\phi$, $\phi_1^{\st}=-\phi_1$, 
while in every parallel holomorphic gauge, $\MD_1=\overline{\partial}_E$ and $\MD_3=\pa_y$, i.e., $A^{(0,1)}= A_y-i\phi_1=0$. 
\end{proposition}
\begin{proof}
With the convention $H(s,s')=\bar{s}^{\top}Hs'$, we compute first in a holomorphic parallel gauge, from the defining equations for the 
$\MD_i'$, that $\bpa H=(\overline{A^{(1,0)}})^{\top}H$ and $\pa_y H=H(-A_y-i\phi_1)$, so in this gauge, $A=A^{(1,0)}=H^{-1}\pa H$ 
and $A_y+i\phi_1=-H^{-1}\pa_yH$.

Suppose next that we know $H$ with respect to a homolomorphic frame. If $g$ is a complex gauge transformation such that $H=g^{\da}g$, 
then in the parallel holomorphic gauge, 
\begin{equation}
A^{(1,0)}=H^{-1}\pa H =g^{-1}(g^{\da})^{-1}(\pa_z g^{\da}) g+g^{-1} \pa_z g, \ \ A^{(0,1)}=0.
\end{equation}
If $\hA$ is the connection form in unitary gauge, then 
\begin{equation}
\hA_z=(g^{\da})^{-1}\pa_zg^{\da}, \ \ \hA_{\bar{z}}=-(\pa_{\bar{z}}g) g^{-1},
\end{equation}
and $\hA_{\bar{z}}^{\da}=-\hA_z$. Thus $g$ transforms the holomorphic form to the unitary one. 

Similarily, the same Higgs field in holomorphic and unitary gauge, $\vp$ and $\phi$, are related by
\begin{equation}
\begin{split}
\phi_z=g\vp g^{-1},\;\phi_{\bz}= (g^{\da})^{-1}\bar{\vp}^{\top} g^{\da}.
\end{split}
\end{equation}

For the final component, suppose that $\MA_y$ is given in holomorphic gauge. Then in unitary gauge, 
\begin{equation}
\begin{split}
A_y=\frac{1}{2}((\pa_yg) g^{-1}-(g^{\da})^{-1}\pa_yg^{\da}),\ \ \phi_1=\frac{i}{2}( (g^{\da})^{-1}\pa_yg^{\da}+\pa_yg^{\da}
(g^{\da})^{-1}).
\end{split}
\end{equation}
\end{proof}

\medskip

We now record some computations in a local holomorphic coordinate chart. 
Writing $\MD_1=\pa_{\bz}+\al$, $\MD_1^{'}=\pa_{z}+A^{(1,0)}$, $\MD_3=\pa_y+\MA_y$ and $\MD_3^{'}=\pa_y+\MA_y^{'}$, we compute:
\begin{equation}
\begin{split}
A^{(1,0)}&=H^{-1}\partial_z H-H^{-1}(\bar{\al})^{\top}H,\\
A&=A^{(1,0)}+\alpha=H^{-1}\partial_z H-H^{-1}\bar{\al}^{\top}H+\al,\\
\vp^{\da}&=H^{-1}\bar{\vp}^{\top}H,\\
\MA_y^{'}&=H^{-1}\partial_y H-H^{-1}\bar{\MA_y}^{\top} H.
\end{split}
\label{99}
\end{equation}
Thus if $\alpha=0$, then $\MA_y=0$, and the adjoint operators become
\begin{equation}
\begin{split}
\MD_1^{\da}=-\MD_1^{'}=-(\pa_z+H^{-1}\pa_z H),\;
\MD_2^{\da}=-\MD_2^{'}=[\vp^{\da},\;],\;
\MD_3^{\da}=-\MD_3^{'}=-\pa_y-H^{-1}\partial_y H.
\end{split}
\end{equation}

Altogether, in a local holomorphic coordinate $z$ for which the metric on $\Si$ equals $g_0^2 |dz|^2$, and in 
the holomorphic parallel gauge where $\MD_1=\bar{\pa}$, $\MD_3=\pa_y$, the \EBE \eqref{algebraEBE} become
\begin{equation}
-\bar{\pa}(H^{-1}\pa H)-g_0^2\pa_y(H^{-1}\pa_y H)+[\vp,\vp^{\star}]=0.
\label{ezHBE}
\end{equation}

Two sets of data $(E,\Theta)$ and $(E, \wt \Theta)$ are called equivalent if there exists a complex gauge transform $g$ such that 
$g^{-1}\wt \MD_i g=\MD_i$, $i=1,2,3$. A key fact is that $(E,\Theta)$ is completely determined by a Higgs pair $(\ME,\vp)$ over the 
Riemann surface $\Si$. 
\begin{proposition}   (1) Suppose that $(E,\Theta)$ and $(E, \wt \Theta)$ are two data sets. 
If the restrictions of $\Theta$ to $E_y$ and $\wt \Theta$ to some possibly different $E_{y'}$ are complex gauge equivalent, 
then $(E,\Theta)$ and $(E, \wt \Theta)$ are equivalent. 

(2) If $(E,\Theta,H)$ is a solution to the \EBE, and if $g$ is a complex gauge transform, then $(E,\Theta^g)$, where $\Theta^g=
(g^{-1}D_1g,g^{-1}D_2g,g^{-1}D_3g), H^g =Hg^{\st_H} g)$ is also a solution.
\end{proposition}
\begin{proof}
Since $\MD_3$ and $\wt \MD_3$ both define isomorphisms of the Higgs pairs,  (1) follows immediately. Then, recalling that
$D_i^{\da}$ is the conjugate of $D_i$ with respect to $H$, one may check (2) directly from the definition.
\end{proof}
\end{subsection}
\end{section}

\begin{section}{Boundary Conditions}
In this section we introduce boundary conditions for the extended Bogomolny equations over $\Sigma\ti \RP$  
at $y = 0$ and as $y\to +\infty$.

\begin{subsection}{$SL(2,\mathbb{R})$ Higgs-bundles}
We impose an asymptotic boundary condition as $y \to +\infty$ by requiring that solutions of \eqref{EBE} converge to 
flat $\SLR$ connections. To explain this more carefully, we recall some basic facts about the moduli space of stable 
$\SLR$ Higgs-bundles, cf.\ \cite{hitchin1987self}, \cite{hitchin1992lie}.  

Consider a Riemann surface $\Sigma$ of genus $g>1$. A Higgs bundle consists of a pair $(\ME,\vp)$ where $\ME$ is a 
holomorphic structure on a complex vector bundle $E$ and $\vp  \in H^{0}(\mathrm{End}(\ME)\otimes K)$ is a Higgs field. 
Let $(\ME,\vp)$ be a rank $2$ Higgs bundle such that $\deg E = 0$. It is proved in \cite{hitchin1987self} that once an 
$\SLR$ structure is fixed, there is an isomorphism $\ME\cong L^{-1}\oplus L$, where $L$ is a line bundle with 
$0 \leq \deg L \leq  g-1$, in terms of which the  Higgs field takes the form 
\begin{equation}
\vp=\begin{pmatrix}  0 & \alpha \\  \beta & 0 \end{pmatrix} 
\end{equation}
where $\al\in H^0(L^{-2}\otimes K)$ and $\be\in H^{0}(L^2\otimes K)$. 
When $\deg L=g-1$ and $L=K^{\frac{1}{2}}$ for one of the $2^{2g}$ square roots of $K$, then we write this canonical form for 
the Higgs field in the familiar form
\begin{equation}
\vp= \begin{pmatrix} 
0 & 1 \\
q & 0
\end{pmatrix}
\end{equation}
Here $1$ is the canonical identity element in $\Hom(L,L^{-1})\otimes K=\Hom(K^{\frac12},K^{-\frac 12})\otimes K=\MO$ and 
$q\in H^{0}(L^{2}\oti K)=H^{0}(K^2)$ is a holomorphic quadratic differential. This set of Higgs bundles constitutes the 
Hitchin component of the $\SLR$ moduli space. 

The splittings with $|\deg L|<g-1$ constitute the non-Hitchin components. Write $k=\deg L$ so that $\deg (L^{-2}\otimes K)=
\deg K-2\deg L=2g-2-2k$. Thus when $0 \leq k<g-1$, the section $\al$ has $2g-2-2k$ zeros; these are of course invariant 
under complex gauge transform.

If $\phi_1 = 0$ in \eqref{EBE}, or if $D_3 = 0$ in \eqref{algebraEBE}, we obtain the Hitchin equation
\begin{equation}
\begin{split}
F_H+[\vp,\vp^{\star}]=0,\;\bar{\pa}_A\vp=0.
\label{Hitchinequation}
\end{split}
\end{equation}

A rank $2$ Higgs pair $(\ME,\vp)$ with $\det(\ME)=\MO$ is stable if for every $\vp$-invariant subbundle $S \subset E$, 
$\deg S<0$. We say in general that $(\ME,\vp)$ is polystable if it is direct sum of stable Higgs bundle. In the rank $2$ case, 
a polystable Higgs bundle takes the form 
$(E=L^{-1}\oplus L,\vp=\begin{pmatrix}  a & 0 \\  0 & -a \end{pmatrix})$, but by  assumption we shall exclude these.

The solvability of the Hitchin equation \eqref{Hitchinequation} was analyzed completely in \cite{hitchin1987self}. 
\begin{theorem}{\cite{hitchin1987self}}
Let $(\ME,\vp)$ be a Higgs pair over $\Sigma$. There exists an irreducible solution $H$ to the Hitchin equations if and only if 
the Higgs pair is stable, and a reducible solution if and only if it is polystable.
\end{theorem}

When $\deg L>0$, the Higgs pairs $(L^{-1}\oplus L,\vp=\begin{pmatrix}
0 & \al \\
\be & 0
\end{pmatrix})$ are all stable. If $\deg L=0$, then $L\cong \MO$ and $E$ is holomorphically trivial. 
If $\vp=\begin{pmatrix} 
0 & \alpha \\
\beta & 0
\end{pmatrix}$, 
then the pair is stable if and only if neither $\al$ nor $\be$ are identically zero. If precisely one of $\al$, $\be$ 
vanishes, the pair is neither stable nor polystable and the Hitchin equation has no solution. If both $\al=\be=0$, 
then the Higgs bundle is polystable and there exist a reducible solution. 

In this paper we restrict attention to irreducible solutions. The moduli space of stable $\SLR$-Higgs pairs can then 
be described as follows:
\begin{theorem}{\cite{hitchin1987self}}
The $\SLR$ Higgs bundle moduli space contains $2g-1$ components, classified by the degree $k$ of the line bundle $L$, $|k|\leq g-1$. 
The component $\MM_k^{\SLR}$ is a smooth manifold of dimension $(6g-6)$ diffeomorphic to a complex vector bundle of 
rank $(g-1+2k)$ over the $2^{2g}$-fold cover of the symmetric product $S^{2g-2-2k}\Sigma$.
\end{theorem}
\proof We sketch the proof. For the $SL(2,\mathbb{R})$ Higgs bundle $(L^{-1}\oplus L,\begin{pmatrix}
0 & \al \\
\be & 0
\end{pmatrix})$, the zeroes of $\al\in H^{0}(L^{-2}\otimes K)$ give a divisor $D$ where $\MO(D)=L^{-2}\otimes K$, 
and hence an element of $S^{2g-2-2k}\Sigma$. Then $\be\in H^{0}(\Sigma,\MO(-D)K^2)$ determines a line bundle. 

Note that since we are working with $SL(2,\mathbb{R})$, given $D$ we can only determine $L^2=\MO(-D)K$, but $L$ itself 
can only be recovered up to the choice of a line bundle $I$ with $I^2=\MO$. There are precisely $2^{2g}$ such choices.
\qed

We recall finally a well-known result:
\begin{proposition}
The harmonic metric $H$ corresponding to a stable $\SLR$ Higgs pair splits with respect to the decomposition $E = L^{-1} \oplus L$,
$H= \begin{pmatrix} h & 0 \\ 0 & h^{-1} \end{pmatrix}$.
\end{proposition}
A proof appears in \cite[Theorem 2.10]{collier2014asymptotics}. 
\end{subsection}

\begin{subsection}{The Nahm Pole Boundary Condition and Holomorphic Data} 
\label{subSectionNahmPole}
We next recall the Nahm pole boundary condition and its associated Hermitian geometry, following \cite{gaiotto2012knot}.

The starting point is the model solution \cite{witten2011fivebranes}. Consider a trivial rank $2$ bundle $E$ over 
$\mathbb{C}\times\RP$. The model Nahm pole solution is
\begin{equation}
A_z=0,\ 
\phi_z=\frac{1}{y} \begin{pmatrix} 0 & 1 \\ 0 & 0 \end{pmatrix}, \ 
\MA_y=-i\phi_1=\frac{1}{2y}
\begin{pmatrix} 1 & 0 \\  0 & -1 \end{pmatrix}. 
\end{equation}
Under the singular complex gauge transformation, these fields become 
$g = \begin{pmatrix}
y^{-\frac{1}{2}} & 0 \\
0 & y^{\frac{1}{2}}
\end{pmatrix}$ to $\vp=\begin{pmatrix} 
0 & 1 \\
0 & 0
\end{pmatrix}$, $A_z = 0$ and $\MA_y = 0$, i.e., the connection in the $\RR^+$ direction transforms to $\del_y$. 

Now, $s=\begin{pmatrix}  ay^{-\frac{1}{2}}\\ by^{\frac{1}{2}} \end{pmatrix}$ is an $\MD_3$ parallel section of $E$ 
for any $a,b \in \RR$, and indeed is a solution of the full extended Bogomolny equations.  A generic solution 
of this form blows up as $\yrz$, but there is a well-defined subbundle $L\subset E$, called the \textbf{vanishing line 
bundle}, defined as the space of solutions which tend to $0$ as $\yrz$.  For this model solution and line bundle,
$\mathrm{span}\, \{\vp(L),L\otimes K\}=E\otimes K$ at all points.  

We say that a solution $(A,\vp,\phi_1)$ to \eqref{EBE} on a rank $2$ Hermitian bundle $E$ with determinant zero over $\Si$
satisfies the Nahm pole boundary condition if in terms of any local trivialization
\begin{equation}
A_z\sim\MO(y^{-1+\ep}),\;
\varphi=\frac{1}{y} \begin{pmatrix}
0 & 1 \\
0 & 0
\end{pmatrix}+\MO(y^{-1+\ep}),\;\MA_y=\frac{1}{2y}\begin{pmatrix}
1 & 0 \\
0 & -1
\end{pmatrix}+\MO(y^{-1+\ep})
\end{equation}
as $y \to 0$.  As described in \cite{MazzeoWitten2013}, it is of course necessary to consider fields which lie in some
function space, e.g. a weighted H\"older space, and the error estimate $\MO(y^{-1 + \ep})$ is interpreted in terms of that norm.
The regularity theory in that paper shows that a solution of the extended Bogomolny equations, or indeed of the full Kapustin-Witten 
system, is then much more regular after being put into gauge. 

In exactly the same way as in the model case, this boundary condition defines a line bundle $L \subset E$, and since $\det E = \MO$, 
we have
$E/L \cong L^{-1}$. On the other hand, $\mbox{span}\{\vp(L), L\otimes K \}=E \otimes K$, so that pushing forward $L$ via
\begin{equation}
L\xrightarrow{\vp} E\otimes K\rightarrow (E/L) \otimes K
\end{equation}
shows that $L\cong L^{-1}\otimes K$, i.e., $L\cong K^{\frac{1}{2}}$, and then $E/L\cong K^{-\frac 12}$. In other words,
\begin{equation}
0\to K^{\frac 12}\to E\to K^{-\frac 12}\to 0.
\end{equation}

In addition, denote $i_1:\vp(L)\rightarrow E\otimes K$ and $i_2:L\otimes K\rightarrow E\otimes K$, and define:
\begin{equation}
\begin{split}
i:\, &\vp(L)\oplus L\otimes K\rightarrow E\otimes K\\
i &=i_1+i_2.
\end{split}
\end{equation}
As $\mbox{span}\{\vp(L), L\otimes K \}=E \otimes K$, we obtain that $i$ is surjective between two rank two bundles thus isomorphism. 
Tensoring by $K^{-1}$, we obtain $E\cong K^{-\frac12}\oplus K^{\frac12}$.

Under a complex gauge transform, we can then put the Higgs field into the form 
$\vp=\begin{pmatrix}
t & 1 \\
\beta' & -t
\end{pmatrix}$. Setting $g=\begin{pmatrix}
1 & 0 \\
-t & 1
\end{pmatrix}$, we compute that $g^{-1}\vp g=\begin{pmatrix} 
0 & 1 \\
\beta & 0
\end{pmatrix}$. This shows that a $SL(2,\mathbb{R})$ Higgs bundle lies in the Hitchin component of the $\SLR$ 
Higgs bundle moduli space. 

In summary, recalling that $\MM^{\EBEt}_{\NP}$ is the moduli space of solutions of the extended Bogomolny equations with limit 
in $SL(2,\mathbb{R})$ and $\MM^{\Fusi}$ is the Hitchin component of stable $SL(2,\mathbb{R})$ Higgs bundle, we have now
explained the map $I_{\NP}:\MM^{\EBEt}_{\NP}\to \MM^{\Fusi}$. Gaiotto and Witten \cite{gaiotto2012knot} conjectured that 
this map is a bijection, and we show below that this is the case. 
\end{subsection}

\begin{subsection}{Knot Singularity}
\label{subsectionKnot}
We next define the model knot singularity introduced by Witten in \cite{witten2011fivebranes}, and the modified Nahm pole
condition for knots.  In the Riemann surface picture, knot singularities correspond to marked points, at which monopoles
are wrapped. 

Fix coordinates $z = x_2 + i x_3 \in \mathbb C$ and $y \in \RP$ on $\mathbb{C}\ti\RP$. 
Then, with respect to the Higgs field $\vp= \begin{pmatrix}
0 & z^n \\
0 & 0
\end{pmatrix}$ and Hermitian metric $H=\begin{pmatrix} 
e^{u} & 0 \\
0 & e^{-u}
\end{pmatrix}$, equation \eqref{ezHBE} takes the form
\begin{equation}
-(\Delta+\pa_y^2)u+r^{2n}e^{2u}=0,
\label{Wittenknotmodel}
\end{equation}
where $\Delta=\pa_{x_2}^2+\pa_{x_3}^2$ and $r=|z|$. 

Assuming homogeneity in $(z,y)$ and radial symmetry in $z$, Witten \cite{witten2011fivebranes} obtained 
the model solution 
\begin{equation}
U_n(r,y)=\log \left( \frac{2(n+1)}{(\sqrt{r^2+y^2}+y)^{n+1}-(\sqrt{r^2+y^2}-y)^{n+1}}\right).
\label{modelsoln}
\end{equation}
To investigate this further, introduce spherical coordinates $(R,\psi,\theta)$, 
\[
R=\sqrt{r^2+y^2}, \ z=re^{i\theta}, \ \sin \psi=\frac{y}{R}, \ \cos \psi=\frac{r}{R}.
\]
Writing $a=\sqrt{r^2+y^2}+y$ and $b=\sqrt{r^2+y^2}-y$, then 
\[
\frac{a}{R}=1+\frac{y}{R}=1+\sin \psi,\ \frac{b}{R}=1-\frac{y}{R}=1-\sin \psi,
\]
and hence 
\[
U_n =-\log y-n\log R+\log \frac{n+1}{S_n(\psi)},
\]
where
\[
S_n(\psi) = \mathcal S_n(a,b)=\sum_{k=0}^na^{n-k}b^k.
\]
Note that $U_0 = -\log y$ when $n=0$, which recovers the model
Nahm pole solution.  Moreover, $U_n$ is compatible with the Nahm pole singularity in the sense that $U_n \sim -\log y$
as $y \to 0$ for $r \geq \epsilon > 0$. 

Defining 
$g_n=\begin{pmatrix}    e^{u_n/2} &  0\\  0 & e^{-u_n/2} \end{pmatrix},$ then in unitary gauge
\begin{equation}
A_z=g_n^{-1}\pa g_n,\;A_{\bz}=-(\bar{\pa}g)g^{-1},\;\phi_z=g_n\vp g_n^{-1},\;\phi_1=\frac{i}{2}(g_n^{-1}\pa_y g_n+\pa_y g_n g_n^{-1}),
\end{equation}
or explicitly, 
\begin{equation}
\begin{split}
\phi_z & =
\begin{pmatrix} 
0 & z^ne^{U_n} \\
0 & 0
\end{pmatrix}   \\
& =\frac{2}{R}\, \frac{(n+1)\cos^n\psi }{(1+\sin\psi)^{n+1}-(1-\sin\psi)^{n+1}} e^{i n\theta} 
\begin{pmatrix}
0 & 1\\
0 & 0
\end{pmatrix}  \\
&   =\frac{1}{R\sin\psi}\frac{(n+1)\cos^n\psi}{S_n(\psi)}e^{in\theta}
\begin{pmatrix}
0 & 1\\
0 & 0
\end{pmatrix}  \\ 
\phi_1&=
-U_n' \begin{pmatrix}
\frac{i}{2} & 0 \\
0 & -\frac{i}{2}
\end{pmatrix}  \\
& =\frac{n+1}{R}\frac{(1+\sin \psi)^{n+1}+(1-\sin \psi)^{n+1}}{(1+\sin \psi)^{n+1}-(1-\sin \psi)^{n+1}}
\begin{pmatrix}
\frac{i}{2} & 0 \\
0 & -\frac{i}{2}
\end{pmatrix} \\ 
A_y & =0.
\end{split}
\end{equation}

Suppose that $s$ is a section with $\MD_3 s=0.$ Then for any $a,b \in \mathbb R$, $s=\begin{pmatrix}
ae^{U_n/2}\\
be^{-U_n/2}
\end{pmatrix}$ 
is a solution, where $e^{U_n}=(n+1)/(yR^nS_n(\psi))$.   As in the Nahm pole case, we can still define a line subbundle $L$ 
corresponding to parallel sections whose limits as $y \to 0$ vanish;  generic parallel sections blow up.  However, a new
feature here is that $\mbox{span}(L \otimes K, \vp(L)) \neq E \otimes K$ precisely at the knot singularities, reflecting
the zeroes of $\vp$.

For any $p \in \Sigma$ we can transport the model solution to $\Si \times \RP$ using the local coordinates $(z,y)$, 
giving an approximate solution $(A^p,\phi^p,\phi_1^p)$ in a neighborhood of $(p,0)$.  It is convenient 
\begin{definition}
A solution $(A,\phi,\phi_1)$ to the \EBE satisfies the general Nahm pole boundary condition with knot singularity of order 
$n$ at $(p,0) \in \Si \times \RP$ if in a suitable gauge it satisfies
\begin{equation}
(A,\phi,\phi_1)=(A^p,\phi^p,\phi_1^p)+\MO(R^{-1+\ep} (\sin \psi)^{-1 + \ep})
\end{equation}
for some $\ep>0$, where $R$ and $\psi$ are the spherical coordiates used above. 
\end{definition}

Corresponding to a solution with knot singularity is a set of holomorphic data. Suppose $(A,\phi,\phi_1)$ is a solution 
with a knot singularity at the points $\{p_j\}$ with orders $n_j$, $j=1,\cdots, N$.  We define the line subbundle $L$ of $E$ and
obtain the exact sequence
\begin{equation}
0\to L \to E \to L^{-1}\to 0.
\label{exactsequecneknot}
\end{equation}
Using the asymptotic boundary condition at $y\to +\infty$ and the Milnor-Wood inequality \cite{milnor1958existence}, 
\cite{wood1971bundles}, we have $|\deg L|\leq g-1$.

The knot singularity and Higgs field induce a map 
\begin{equation}
P:L\xrightarrow{\vp} E\otimes K\rightarrow L^{-1}\otimes K.
\label{inducedmapP}
\end{equation}
Regarding $P$ as an element of $H^0(L^{-2}\otimes K)$, we deduce that that there are $2g-2-2\deg L$ marked points, counted
with multiplicity.  

The data we must specify then consists of the following:
\begin{enumerate}
\item An $SL(2;\mathbb{C})$ Higgs bundle with a line subbundle $L$;
\item Marked points $\{p_j\}$ with orders $n_j$;
\item Generic parallel sections of $E$ over $\Si\setminus \{p_j\}$ blow up at the rate $y^{-\frac{1}{2}}$;
\item The section $P\in H^0(L^{-2}K)$ in \eqref{inducedmapP} has zeroes precisely at $p_j$ of order $n_j$.
\end{enumerate}

Just as for the Nahm pole case, we impose an $SL(2,\mathbb{R})$ structure on the Higgs bundle. The following
assumption simplifies the Hermitian geometric data.
\begin{definition}
\label{compatiblewithSL(2R)limit}
Suppose we have a solution to \eqref{EBE} which satisfies the general Nahm pole boundary conditions, and assume that 
the solution converges to an $SL(2,\mathbb{R})$ Higgs bundle $(L^{-1}\oplus L, \vp=\begin{pmatrix}
0 & \al \\
\be & 0
\end{pmatrix})$ as $y \to \infty$. We say that this solution is compatible with the $SL(2,\mathbb{R})$ structure at $y=\infty$ if 
either $L$ or $L^{-1}$ is the vanishing line bundle. 
\end{definition}
Merely assuming that the Higgs bundle converges to an $SL(2,\mathbb{R})$ Higgs bundle, as above, 
is not enough to imply that $L$ is the vanishing line bundle. 

\begin{remark}
\label{sl2remark}
If the exact sequence \eqref{exactsequecneknot} splits, the Higgs field may take the slightly more general form $\vp=\begin{pmatrix}
t & \al \\
\be & -t
\end{pmatrix}$. Such Higgs fields with $t \neq 0$ exist, but at present we do not know whether it is possible to solve
the \EBE with knot singularity with this data. The vanishing of $t$ will play a minor but important technical role below 
in Proposition~\ref{Knotsingularityexpansion}, which we need in proving uniqueness theorems later.
\end{remark}

The compatibility of the solution with the $SL(2,\mathbb{R})$ structure is a technical condition that allows us to reduce the Bogomolny 
equation to a scalar equation. There is one special case where we do not need to assume this compatibility condition. Under the 
assumption of Definition \ref{compatiblewithSL(2R)limit}, denote the vanishing line bundle as $L'$. We then obtain 
\begin{proposition}
\label{vanishinglinebundle}
If $L'\neq L$ or $L^{-1}$, then $\deg L'\leq-|\deg L|$,
\end{proposition}
\begin{proof}
The line subbundle $L'$ induces the exact sequence:
\begin{equation*}
0\to L'\to L^{-1}\oplus L\to L'^{-1}\to 0,
\end{equation*}
which defines the holomorphic map $\gamma_1:L\to L'^{-1}$ and $\gamma_2:L^{-1}\to L'^{-1}$.  Since $L'\neq L$ or $L^{-1}$, 
we obtain that neither $\gamma_1$ nor $\gamma_2$ equal the identity. In other words, we obtain non-zero elements $\gamma_1\in 
H^0(L^{-1}\otimes L'^{-1})$ and $\gamma_2\in H^0(L\otimes L'^{-1})$. Since $\gamma_1$, $\gamma_2$ do not have poles, we 
obtain $\deg(L^{-1}\otimes L'^{-1})\geq 0$ and $\deg(L\otimes L'^{-1})\geq 0$, which implies $\deg L'\leq -|\deg L|.$
\end{proof}
Denoting by $N:=\sum n_j$ the sum of the orders of the marked points, we conclude the
\begin{corollary}
If $\deg L>0$ and $N< 2g-2+2\deg L$, then $L'=L$.
\end{corollary}
\begin{proof}
Recall that $N=2g-2-2\deg L'$, and furthermore, if $N< 2g-2+2\deg L$, then $\deg L'> -\deg L$. Proposition \ref{vanishinglinebundle}
then implies this result. 
\end{proof}
\end{subsection}

\begin{subsection}{Regularity}
We have defined these boundary conditions both at $y=0$ and at the knot singular points by requiring the fields $(A, \phi)$ 
to differ from the corresponding model solutions by an error term, the relative size of which is smaller than the model.
In the existence theorems later in this paper this may be all we know about solutions at first. However, to be able to carry
out many further arguments it is important to know that, in an appropriate gauge, solutions have much stronger regularity 
properties.  Fortunately there is an appropriate regularity theory available which was developed in \cite{MazzeoWitten2013} in 
the Nahm pole case and \cite{MazzeoWitten2017} near the knot singularities.  We note that in those papers solutions to the full
four-dimensional KW system are treated, but those results specialize directly to the present setting, and in fact there
are some minor but important strengthenings here which we point out inter alia. 

Regularity theory relies on ellipticity, and to turn the \EBE into an elliptic system we must add an appropriate gauge condition.
We recall the choice made in \cite{MazzeoWitten2013} for the KW system on a four-manifold and then specialize it in
our dimensionally reduced setting. Fix a pair of fields $(\widehat{A}^0, \widehat{\phi}^0)$ on a four-manifold which are 
either solutions or approximate solutions of KW equations. Then nearby fields can be written in the form $(\widehat{A}, 
\widehat{\phi}) = (\widehat{A}^0, \widehat{\phi}^0) + (\alpha, \psi)$.  The gauge-fixing equation is then
\begin{equation}
d_{\widehat{A}^0}^* \alpha + \star [ \widehat{\phi}^0, \star \psi] = 0.
\label{gaugeeqn}
\end{equation}
It is shown in \cite{MazzeoWitten2013} that adjoining \eqref{gaugeeqn} to the KW equations is elliptic. 

Denote by $\calL$ the linearization of this system at $(\widehat{A}^0, \widehat{\phi}^0)$. This is a Dirac-type
operator with coefficients which blow up at $y=0$ and $R=0$ in a very special manner.  In the absence of knots,
$\calL$ is (up to a multiplicative factor) a {\it uniformly degenerate} operator, while near a knot it lies in
a slightly more general class of incomplete iterated edge operators.  These are classes of degenerate differential
operators for which tools of geometric microlocal analysis may be applied to construct parametrices, which
in turn lead to strong mapping and regularity properties.  We refer to \cite{MazzeoWitten2013}, \cite{MazzeoWitten2017}
for further discussion about all of this and simply state the consequences of this theory here. 

Before doing this we first recall that for degenerate elliptic problems it is too restrictive to expect solutions to
be smooth up to the boundary.  Instead we consider polyhomogeneous regularity.  Let $X$ be a manifold
with boundary, with coordinates $(s,z)$ near a boundary point, with $s \geq 0$ and $z$ a coordinate in the boundary. 
We say that a function $u$ is polyhomogeneous at $\pa X$ if
\[
u(s,z) \sim \sum_{j=0}^\infty \sum_{\ell=0}^{N_j}  a_{j\ell}(z) s^{\gamma_j} (\log s)^\ell, \ \ a_{j \ell} \in \calC^\infty(\pa X).
\]
The exponents $\gamma_j$ here is a sequence of (possibly complex) numbers with real parts tending to infinity; 
importantly, for each $j$, only finitely many factors with (positive integral) powers of $\log s$ can appear. 
The set of pairs $(\gamma_j, \ell)$ which appear in this expansion is called the index set for this expansion.
Denoting this index set by $\mathcal I$, we say that $u$ is $\mathcal I$-smooth, which emphasizes that this
regularity is a very close relative of and satisfactory replacement for ordinary smoothness. 
Similarly, if $X$ is a manifold with corners of codimension $2$, with coordinates $(s_1, s_2, z)$ near a point on the corner,
then $u$ is polyhomogeneous if 
\[
u(s_1, s_2, z) \sim \sum_{i, j=0}^\infty \sum_{p, q = 0}^{N_{i,j}}a_{ijpq}(z) s_1^{\gamma_i} s_2^{\lambda_j} (\log s_1)^p (\log s_2)^q.
\]
In other words, we require the expansion for $u$ to be of product type near the corner.  These are all classical expansions
with the usual meaning and the corresponding expansions for any number of derivatives hold as well. 
The reason for introducing this more general notion is precisely because at least in favorable situations, 
solutions of have this regularity but are not smooth in a classical sense.  The important point is that this
is a perfectly satisfactory replacement for smoothness up to the boundary and allows one to analyze and
manipulate expressions using these `Taylor series' type expansions.  

We first consider the case where there are no knot singularities, but note that this result is a local one and
can be applied away from knot singular points.  Here the manifold with boundary is simply $\Si \ti \RP$
and we use coordinates $(y,z)$.
\begin{proposition}[\cite{MazzeoWitten2013}]
\label{Nahmpoleexpansion}
Let $(A, \vp, \phi_1)$ be a solution to the \EBE near $y=0$ which satisfies the Nahm pole boundary conditions
and is in gauge relative to the model approximate solution.  Then these fields are polyhomogeneous with
\[
A=\MO(1), \ \varphi 
= \frac{1}{y}\begin{pmatrix} 0 & 1 \\ 0 & 0 \end{pmatrix} + \MO(y), \ \phi_1 
= \frac{1}{y}\begin{pmatrix} \frac{i}{2} & 0 \\ 0 & -\frac{i}{2} \end{pmatrix} + \MO(y\log y)\ 
\]
\end{proposition}
This statement incorporates recent work in \cite{HeMikhaylov2017} which provides much more detail
about the expansions than is present in \cite{MazzeoWitten2013}.

To state the corresponding result in the presence of a knot singularity, we first define the manifold with
corners $X$ to be the blowup of $\Si \ti \RP$ around each of the knot singular points $(p_j, 0)$. In other
words, we replace each $(p_j, 0)$ by the hemisphere $R=0$ (parametrized by the spherical coordinate
variables $(\psi, \theta)$), points of which label directions of approach to that point.  The discussion is
local near each $p_j$ so we may as well fix coordinates $(R, \psi, \theta)$. The corner of $X$ is defined by
$R=\psi = 0$.
\begin{proposition}[\cite{MazzeoWitten2017}]
\label{Knotsingularityexpansion}
Let $(A, \phi, \phi_1)$ satisfy the \EBE near $(0,0)$ as well as the gauge condition relative to the
model knot solution $U_{n}$.  Then these fields are polyhomogeneous with the same asymptotics 
as in the previous proposition when $y \to 0$ away from the knot, while 
\[
A = A^n + \MO(R^\epsilon \sin \psi),\ \varphi = \varphi^n + \MO(R^\epsilon \sin \psi),\ \phi_1 = \phi_1^n + 
\MO(R^{\epsilon} (\sin \psi) \log(\sin \psi))
\]
near the knot.  Here $(A^n, \varphi^n, \phi_1^n)$ is the model solution described in \S 3.3 associated to $U_n$.
\end{proposition}
Referring to the language of \cite{MazzeoWitten2017}, these rates of decay, i.e., the first exponents
in the expansions beyond the initial model terms, are indicial roots of type $II$ and $II'$. The exponent $0$
is a possible indicial root of type $II'$, but does not appear in our setting because the $\mathrm{SL}(2,\RR)$ 
structure forces $\varphi$ to have no diagonal terms, see Remark~\ref{sl2remark}, and it is precisely
in these diagonal terms where the exponent $0$ might appear in the expansion.
\end{subsection}

\begin{subsection}{The Boundary Condition for the Hermitian Metric}
Since we must deal with singularities of the gauge field, it is often simpler to work in holomorphic gauge but
consider singular Hermitian metrics. We now describe a boundary condition for the Hermitian metric compatible 
with the unitary boundary condition defined above.  We use the Riemannian metric $g=g_0^2|dz|^2+dy^2$ on 
$\Sigma\ti\RP$.   The following result is a direct consequence of the previous computations in 
Section \ref{subSectionNahmPole}, \ref{subsectionKnot}.
\begin{proposition}
Consider the Higgs bundle $(E\cong L^{-1}\oplus L,\vp=\begin{pmatrix}
0 & \al \\
\be & 0
\end{pmatrix})$. Fix $p\in\Sigma\ti\{0\}$ and an open set $U_p$ containing $p$. Let $H$ be a polyhomogeneous solution 
to the Hermitian Extended-Bogomolny Equations \eqref{algebraEBE}.

(1) Suppose that in a local trivilization on $U_p$, $\vp|_{U_p}=\begin{pmatrix}
0 & 1 \\
\star & 0
\end{pmatrix}$. If for some $\ep>0$, 
\begin{equation}
\begin{split}
H\sim\begin{pmatrix}
y^{-1}(g_0+\MO(y^{\ep})) & 0 \\
0 & y(g_0^{-1}+\MO(y^{\ep}))
\end{pmatrix} \quad \mbox{as}\ y \to 0,
\end{split}
\end{equation}
then the unitary solution with respect to $H$ satisfies the Nahm pole boundary condition near $p$ and $\begin{pmatrix}
0\\
1
\end{pmatrix}$ is the vanishing line bundle in this trivialization. 

(2) Suppose that in a local trivialization on $U_p$, $\vp|_{U_p}=\begin{pmatrix}
0 & z^n \\
\star & 0
\end{pmatrix}$ (where $z=0$ is the point $p$). In spherical coordinates $(R,\theta, \psi)$, suppose for some $\ep>0$,
\begin{equation}
H = \begin{pmatrix}
e^{U_n}(1+\MO(R^{\ep})) & 0 \\
0 & e^{-U_n}(1+\MO(R^{\ep}))
\end{pmatrix} \ \mbox{as}\ R \to 0.
\end{equation}
Then the unitary solution with respect to $H$  satisfies the Nahm pole condition with knot singularity at $p$ and $\begin{pmatrix}
0\\
1
\end{pmatrix}$ is the vanishing line bundle in this trivialization. .
\end{proposition}

Since we wish to work with holomorphic gauge fields and singular Hermitian metrics, we obtain some restrictions. 
Let $P$ be an $SU(2)$ bundle and $(A,\phi,\phi_1)$ a solution to the Extended Bogonomy Equations \eqref{EBE} with 
Nahm pole boundary and knot singularities of order $n_j$ at the points $p_j$, $j=1,\cdots,n$. For each $j$ choose small balls 
$B_j$ around $p_j$, and also let $B_0$ be a neighborhood of $\Sigma\setminus\{B_1,\cdots,B_k\}$ which does not contain any 
of the $p_j$. Choosing a partition of unity $\chi_j$ subordinate to this cover, define the approximate solution 
$u=\sum_{j=0}\chi_j U_{n_j}$ where $U_{n_j}$ is the model solution, and with $U_{n_0}=-\log y$. 
\begin{proposition}
\label{metricaprior}
There exists a Hermitan bundle $(E,H)$ such that: 

(1) $(H,A^{(0,1)},\vp,\MA_y)$ is a solution to the Hermitian Extended Bogomolny equations;

(2) $(A^{(0,1)},\vp,\MA_y)$ is bounded as $y\to 0$;

(3) $H=\begin{pmatrix}
e^{u}h_{11} & h_{12} \\
h_{21} & e^{-u}h_{22}
\end{pmatrix}$, where $u$ is the approximate function above and the $h_{ij}$ are bounded. 
\end{proposition}
\proof  We have explained that $(A,\phi,\phi_1)$ is polyhomogeneous, i.e., $(A,\phi,\phi_1)= (A^{p_j},\phi^{p_j},\phi_1^{p_j})+(a,b,c)$
near $p_j$, where $(a,b,c)$ are bounded. Near other points of $\Sigma \ti \{0\}$ $(A,\phi,\phi_1)$ is the sum of a Nahm pole
and a bounded term. Since $P$ is an $SU(2)$ bundle over $\Sigma\ti \RP$, it is necessarily trivial, so consider the associated
rank $2$ Hermitian bundle $(E,H_0)$, with $H_0 = \mathrm{Id}$ in some trivialization. Now write $H=
\begin{pmatrix}
h_{11} & h_{12} \\
h_{21} & h_{22}
\end{pmatrix}$ where the $h_{ij}$ are bounded. Then $(H_0,A^{(0,1)},\vp,\MA_y)$, where $\vp=\phi_z$, $\MA_y=A_y-i\phi_1$, 
is a solution to the Hermitan extended Bogomolny equations \eqref{algebraEBE}. Consider the complex gauge transform 
$g=\begin{pmatrix}
e^{\frac u2} & 0 \\
0 & e^{-\frac{u}{2}}
\end{pmatrix}$. Since $u$ is compatible with the knot singularity, we obtain a new solution
$(H',A^{(0,1)'},\vp',\MA_y')$, $H'=H_0g^{\da}g=\begin{pmatrix}
e^{u}h_{11} & h_{12} \\
h_{21} & e^{-u}h_{22}
\end{pmatrix}$, and $A^{(0,1)'},\vp',\MA_y'$ are all bounded. 
\qed

\medskip

We conclude this section with a brief discussion about the regularity of a harmonic metric which satisfies the boundary
conditions described here.  Such metrics correspond precisely to the solutions $(A_z, A_y ,\varphi,\phi_1)$ of the original 
\EBE, and for this reason one obvious route to obtain this regularity is to exhibit the direct formula from the  set of
$A's$ and $\phi's$ to the metric $H$.  Another reasonable approach is to simply look at the equation \eqref{ezHBE} 
defining $H$ and prove the necessarily regularity directly from this equation.    In fact, the methods used
in \cite{MazzeoWitten2013} and \cite{MazzeoWitten2017} are sufficiently robust that this adaptation
is quite straightforward.  In the interests of efficiency, we simply state the conclusion:
\begin{lemma}
A harmonic metric $H$ which satisfies the boundary conditions discussed above is necessarily polyhomogeneous.
\end{lemma}
The terms which appear in the polyhomogeneous expansion of $H$ may be determined by the obvious formal calculations 
once we know that the expansion actually exists. 
\end{subsection}

\end{section}

\begin{section}{Existence of Solutions}
We shall prove in this section an existence theorem for the \EBE on $\Si \ti \RP$, either without or with knot singularities at $y=0$. 
The proofs employ the classical barrier method, which we review briefly.
\begin{subsection}{Semilinear Elliptic Equations on Noncompact Manifolds}
We consider on a Riemannian manifold $(W,g)$ the elliptic equation
\begin{equation}
N(u) := -\Delta u+F(x,u)=0, \ \ F \in \calC^\infty(W \times \mathbb R).
\label{Yamb}
\end{equation}
A $\MC^2$ function $u^+$ is called a supersolution for this problem if $N(u^+) \geq 0$, while $u^-$ is called a subsolution
if $N(u^-) \leq 0$. These are called {\it  barriers} for the operator. It is often much simpler to construct such functions which are 
only continuous, and which satisfy the corresponding differential inequalities weakly (either in the distributional or viscosity sense). 
\begin{proposition}
Suppose that $W$ is a possibly open manifold, and that there exist continuous barriers $u^\pm$ which satisfy
$u^- \leq u^+$ everywhere on $W$. Then there exists a solution $u$ to $N(u) = 0$ which satisfies  $u^- \leq u \leq u^+$. 
\label{KWE}
\end{proposition}
\proof  (Sketch)  We first assume that $W$ is a compact manifold with boundary. Then $u^\pm$ are bounded functions
and we may choose $\lambda > 0$ so that $\pa_u F(x,u) \leq \lambda$ for all numbers $u$ lying in the interval
$[u^-(x), u^+(x)]$ for every $x \in W$. The equation can then be written as
\[
(\Delta - \lambda) u = \widetilde{F}(x,u) := F(x,u) - \lambda u.
\]
We then define a sequence of functions $u_j$, $j = 0, 1, 2, \ldots$, by setting $u_0 = u^-$ and successively solving
$(\Delta - \lambda) u_{j+1} = \widetilde{F}(x, u_j)$, and with $u_{j+1}$ equal to some fixed function $\psi$ on $\pa W$
which satisfies $u^-|_{\pa W} \leq \psi \leq u^+|_{\pa W}$. The monotonicity of $\widetilde{F}$ in $u$ and the maximum principle
can be used to prove inductively that $u^- = u_0 \leq u_1 \leq u_2 \leq \cdots \leq u^+$.  When $W$ is a manifold
with boundary we require a version of the maximum principle which holds up to the boundary even for weak solutions;
one version appears in \cite[ Theorem II.1]{ishii1990viscosity}. 

It is then obvious that $u_j$ converges pointwise to an $L^\infty$ function $u$, and standard elliptic regularity implies
that $u \in \MC^\infty$ and that $N(u) = 0$. 

Now suppose that $W$ is an open manifold.  Choose a sequence of compact smooth manifolds with boundary $W_k$
with $W_1 \subset W_2 \subset \cdots$, which exhaust all of $W$.  For each $k$, choose a function $\psi_k$ on $\pa W_k$
which lies  between $u^-$ and $u^+$ on this boundary, and then find a solution $u_k$ to $N(u_k) = 0$
on $W_k$, $u_k = \psi_k$ on $\pa W_k$.  The sequence $u_k$ is uniformly bounded on any compact subset of $W$, so we may 
choose a sequence which converges (by elliptic regularity) in the $\MC^\infty$ topology on any compact subset of $W$. The limit 
function is a solution of $N$ and still satisfies $u^- \leq u \leq u^+$ on all of $W$. 
\qed

We conclude this general discussion by making a few comments about the construction of weak barriers.  A very convenient principle is that sub- and 
supersolutions may be constructed locally in the following sense.  Suppose that $\calU_1$ and $\calU_2$ are two open sets in $W$ and that
$w_j$ is a supersolution for $N$ on $\calU_j$, $j = 1, 2$.  Define the function $w$ on $\calU_1 \cup \calU_2$ by setting $w = w_1$ on
$\calU_1 \setminus (\calU_1 \cap \calU_2)$, $w = w_2$ on $\calU_2 \setminus (\calU_1 \cap \calU_2)$, and $w = \min\{w_1, w_2\}$ on
$\calU_1 \cap \calU_2$. Then $w$ is a supersolution for $N$ on $\calU_1 \cup \calU_2$.  Similarly, the maximum of two (or any finite number)
of subsolutions is again a subsolution.   In our work below, the individual $w_j$ will typically be smooth, but the new barrier $w$ produced
in this way is only piecewise smooth, but is still a sub- or supersolution in the weak sense.  We refer to \cite[Appendix A]{ChruscielMazzeo} for a proof. 
\end{subsection}

\begin{subsection}{The Scalar Form of the Extended Bogomolny Equations}
Following the discussion in \S 3, suppose that $E \cong L \oplus L^{-1}$ and
\begin{equation}
\vp=\begin{pmatrix}
0 & \alpha \\
\beta & 0
\end{pmatrix}. 
\label{nonHitchin}
\end{equation}
When $\deg L = g-1$, $L = K^{1/2}$ and $\alpha = 1$, we seek a solution of the \EBE which satisfies the Nahm pole boundary
condition at $y=0$, while if $\deg L < g-1$, then the zeroes of $\alpha$ determine points and multiplicities $p_j$ and $n_j$ on $\Si$
at $y=0$ and we search for a solution which satisfies the Nahm pole boundary condition with knot singularities at these points. 

Fix a metric $g=g_0^2|dz|^2+dy^2$ on $\Si \ti \RP$ (where $z=x_2+ix_3$ is a local holomorphic coordinate on $\Si$), and assume also that 
the solution metric splits as 
$H= \begin{pmatrix}
h & 0 \\
0 & h^{-1}
\end{pmatrix}$, where $h$ is a bundle metric on $L^{-1}$.  We are then looking for a solution to 
\begin{equation}
\begin{split}
-\Delta_{g} \log h + g_0^{-2}(h^2\alpha \bar{\al}-h^{-2}\be \bar{\be})=0.
\label{localequation}
\end{split}
\end{equation}
We simplify this slightly further. Choose a background metric $h_0$ on $L^{-1}$ and recall that its curvature equals $-\Delta_{g_0} \log h_0$.
Then writing $h=h_0e^u$ and calculating the norms of $\alpha$ and $\beta$ in terms of $g_0$ and $h_0$, \eqref{localequation} becomes
\begin{equation}
K_{h_0}-(\Delta_{g_0}+\pa_y^2) u+|\al|^2e^{2u}-|\be|^2e^{-2u}=0.
\label{scaleEBE}
\end{equation}
In the remainder of this paper, we denote by $N(u)$ the operator on the left in \eqref{scaleEBE}.

An explicit solution to this equation was noted by Mikhaylov in a special case \cite{MikhaylovVpersonal}:
\begin{example}
Consider the Higgs pair $(E\cong K^{\frac 12}\oplus K^{-\frac 12},\vp= \begin{pmatrix}
0 & 1 \\
0 & 0
\end{pmatrix})$. Let $g_0$ be the hyperbolic metric on $\Sigma$ with curvature $-2$ and $h_0$ the naturally induced metric on $K^{-1/2}$, for 
which $K_{h_0}=-1$. Then restricted to $\Si$-independent functions, \eqref{scaleEBE} equals
\begin{equation}
\begin{split}
-1-\pa_y^2 u+e^{2u}=0.
\end{split}
\label{ODE}
\end{equation}
We seek a solution for which $u\sim -\log y$ as $y\to 0$ and $v\to 0$ as $y\to \infty$. The first integral of \eqref{ODE} is $u'=-\sqrt{e^{2u}-2u-1}$, and
hence the unique solution is
\begin{equation}
\int_{u}^{\infty}\frac{ds}{\sqrt{e^{2s}-2s-1}}=y.
\end{equation}
Note that $u$ is monotone decreasing and strictly positive for all $y > 0$. 

We now describe the precise asymptotics of this solution.  If $u \to \infty$, then $s$ is large; write the denominator as $e^s \sqrt{1 - (2s+1)e^{-2s}}$, 
whence 
\[
y = \int_u^\infty e^{-s} (1 + \tfrac12 (2s+1) e^{-2s} + \ldots) \, ds \sim e^{-u} + \ldots,
\]
so $u \sim -\log y$. Similarly, if $u < \epsilon$ for some small $\epsilon$, then $e^{2s} - 2s-1 \sim 2s^2 + \ldots$ when $u < s < \epsilon$, so
\[
u = \int_\epsilon^\infty \frac{ds}{\sqrt{e^{2s}-2s-1}} + \int_u^\epsilon (\tfrac{1}{\sqrt{2}s} + \ldots )\, ds = A - \tfrac{1}{\sqrt{2}}\log u + \ldots,
\]
so $u = Ce^{- \sqrt{2} y} + \ldots$.  Obviously, with only a little more effort, one may develop full asymptotics in both regimes. 
\end{example}
\end{subsection}

\begin{subsection}{Limiting Solution at Infinity}
We first consider the simpler problem of finding a solution of the reduction of \eqref{scaleEBE} reduced to $\Si$, i.e., of 
\begin{equation}
K-\Delta u_\infty +|\al|^2e^{2u_\infty}-|\be|^2e^{-2u_\infty}=0,
\label{102}
\end{equation}
where $K = K_{h_0}$ and $\Delta = \Delta_{g_0}$. Without loss of generality, we assume $\deg L\geq 0$ and note that since $\deg L^{-1} \leq 0$, $\int_\Si K \leq  0$ (and is strictly
negative if the degree of $L$ is positive). A solution to \eqref{102} is the obvious candidate for the limit as $y \to \infty$ of solutions 
on $\Si \ti \RP$.   

\begin{proposition}
If $\alpha \not\equiv 0$, which is equivalent to the stability of the pair $(E, \vp)$, there exists a solution $u_\infty \in \calC^\infty(\Si)$ to \eqref{102}.
\end{proposition}
\proof 
Since this is an equation on $\Si$ rather than $\Si \ti \RP$, this follows immediately from the existence of solutions to the 
Hitchin equations \cite{hitchin1987self}.  However, we give another proof, at least when $\deg L>0$, using the barrier method. 
A proof in the same style when $\deg L=0$ requires more work so we omit it. 

Solve $\Delta w^- = K - \overline{K}$, where $\overline{K} < 0$ is the average of $K$, and set $u^-=w^- - A$ for some constant $A$.  
Then $K-\Delta u^-+ |\al|^2e^{2u^-}-
|\be|^2e^{-2u^-} \leq \overline{K} +|\al|^2e^{w^--A}$, which is negative when $A$ is sufficiently large.  Thus $u^-$ is a subsolution. 

To obtain a supersolution, first modify the background metric $h_0$ by multiplying it by a suitable positive factor so that its curvature $K$ is 
positive near the zeroes of $\alpha$. Next solve $\Delta w^+ =|\al|^2- B$ where $B$ is the average of $|\al|^2$ and set $u^+= w^+ +A$. Then 
\begin{multline*}
K-\Delta u^++|\al|^2e^{2u^+}-|\be|^2e^{-2u^{+}} = K + B+|\al|^2 (e^{2(w^+ + A)}-1)-|\be|^2e^{-2(w^+ + A)} \\ \geq 
K + B + 2|\alpha|^2(w^+ + A) - |\beta|^2 e^{-2(w^+ + A)}.
\end{multline*}
Away from the zeroes of $\alpha$ this is certainly positive if we choose $A$ sufficiently large. Near these zeroes we obtain
positivity using that $K + B > 0$ there and since the final term can be made arbitrarily small.
Thus $u^+$ is a supersolution.   

Noting that $u^- < u^+$ and applying Proposition~\ref{KWE}, we obtain a solution of \eqref{102}.
\qed

Observe that since it is only the boundary condition, but not the equation, which depends on $y$, this limiting solution
is actually a solution of \eqref{scaleEBE} on any semi-infinite region $\Si \ti [y_0, \infty)$, $y_0 > 0$. 
\end{subsection}

\begin{subsection}{Approximate solutions and regularity near $y=0$}
As a complement to the result in the previous subsection, we now construct an approximate solution $u_0$ to \eqref{scaleEBE} near $\{y=0\}$. 
Unlike there, however, we do not find an exact solution, but rather show how to build an initial approximate solution and then incrementally
correct it so that it solves \eqref{scaleEBE} to all orders as $y \to 0$.   In the next subsection we use $u_0$ and $u_\infty$ together to construct
global barriers.

We first begin with the simpler situation where there is only a Nahm pole singularity without knots.
\begin{proposition}
Let $L = K^{1/2}$ and $\alpha \equiv 1$.  Then there exists a function $u_0$ which is polyhomogeneous as $y \to 0$ and is such that $N(u_0) = f$
decays faster than $y^\ell$ for any $\ell \geq 0$.
\end{proposition}
\proof  We seek $u_0$ with a polyhomogeneous expansion of the form 
\[
-\log y + \sum_{j, \ell} a_{j \ell}(z) y^j (\log y)^\ell := -\log y + v,
\]
where all the coefficients are smooth in $z$, and where the number of $\log y$ factors is finite for each $j$.  Rewriting $N(-\log y + v)$ as 
\begin{equation}
\left(-\del_y^2 + \frac{2}{y^2}\right) v + \frac{1}{y^2}(e^{2v} - 2v-1) - |\beta|^2 y^2 e^{-2v} - \Delta_{g_0} v + K_{h_0},
\label{modpsi}
\end{equation}
and inserting the putative expansion for $v$ shows that $a_{0\ell} = a_{1\ell} = 0$ for all $\ell$ and $a_{21} = \tfrac13( K_{h_0} - |\beta|^2$, $a_{2\ell} = 0$
for $\ell > 1$, i.e.,  $v \sim a_{21} y^2 \log y + a_{20}y^2 + \MO(y^3(\log y)^\ell)$ for some $\ell$. Inductively we can solve for each of the coefficients
$a_{j\ell}$ with $j > 2$ using that 
\begin{multline*}
(-\del_y^2 + 2/y^2) y^j (\log y)^\ell \\
=  y^{j-2} (\log y)^{\ell-2} \left((-j(j-1)+2) (\log y)^2 - \ell(2j-1) \log y- \ell(\ell-1) \right).
\end{multline*}
Note that the coefficient $a_{20}$ is not formally determined in this process and different choices will lead to different formal expansions, and also that
there are increasingly high powers of $\log y$ higher up in the expansion. 

Now use Borel summation to choose a polyhomogeneous function $u_0$ with this expansion.  This has a Nahm pole at $y=0$ and
satisfies $N(u_0) = f = \MO(y^\ell)$ for all $\ell$, as desired. 
\qed

\medskip

We next turn to the construction of a similar approximate solution to all orders in the presence of knot singularities. To carry this out,
we first review a geometric construction from \cite{MazzeoWitten2017} which is at the heart of the regularity theorem quoted
in \S 3.4 for the full \EBE and the analogous result for \eqref{scaleEBE} which we describe below.   

If $p \in \Si$, we define the blowup of $\Si \ti \RP$ at $(p,0)$ to consist of the disjoint union $(\Si \times \RP) \setminus \{(p,0)\}$ and
the hemisphere $S^2_+$, which we regard as the set of inward-pointing unit normal vectors at $(p,0)$, and denote by $[ \Si \ti \RP ; \{(p,0)\}]$,
or more simply, just $(\Si\ti\RP)_p$ There is a blowdown map which is the identity away from $(p,0)$ and maps the entire hemisphere to this 
point. This set is endowed with the unique minimal topology and differential structure so that the lifts of smooth functions on
$\Si \ti \RP$ and polar coordinates around $(p,0)$ are smooth.  We use spherical coordinates $(R, \psi, \theta)$ around this point,
so $R=0$ is the hemisphere and $\psi = 0$ defines the original boundary $y=0$ away from $R=0$.    This is a smooth manifold
with corners of codimension two. 

Now fix a nonzero element $\alpha \in H^0(L^{-2}K)$ and denote its divisor by $\sum_{j=1}^N n_jp_j$.  For each $j$, choose a small ball $\hB_j$ 
and a local holomorphic coordinate $z$ so that $p_j=\{z=0\}$, and write $|\al|^2=  \sigma_j^2r^{2n_j}$ there, with $r = |z|$ and $\sigma_j>0$.  
Extend $r$ from the union of these balls to a smooth positive function on $\Si \setminus \{p_1, \ldots, p_N\}$.   By the existence of isothermal
coordinates, we write $g_0 = e^{2\phi}\bar{g}_0$ where $\bar{g}_0$ is flat on each $\hB_j$, and set $g = g_0 + dy^2$, $\bar{g} = \bar{g}_0 + dy^2$. 
Then $\Delta_{g_0} = e^{-2\phi} \Delta_{\bar{g}_0}$ in these balls, and by dilating $\bar{g}_0$, we can assume that $e^{-2\phi} = 1$ at each $p_j$. 
We denote by $(\Si \ti \RP)_{p_1, \ldots, p_N}$ the blowup of $\Si \ti \RP$ at the collection of points $\{p_1, \ldots, p_N\}$.

\begin{proposition}
With all notation as above, there exists a function $u_0$ which is polyhomogeneous on $(\Si \ti \RP)_{p_1, \ldots, p_N}$ and which satisfies
$N(u_0) = f$ with $f$ smooth and vanishing to all orders as $y \to 0$ (i.e., at all  boundary components of the blowup.
\end{proposition}
\proof
In a manner analogous to the previous proposition, we construct a polyhomogeneous series expansion for $u_0$ term-by-term, but now at each of
the boundary faces of $(\Si \ti \RP)_{p_1, \ldots, p_N}$.  

The initial term of this expansion involves the model solutions $U_n$. Choose nonintersecting balls $\hB_j$ with $B_j\subset\subset \hB_j$ and an 
open set $\hB_0 \subset \Sigma\setminus \cup_{j=1}^N \overline{B}_j$ so that $\cup_{j=0}^N \hB_j=\Sigma$. Let $\{\chi_j\}$  be a partition of unity 
subordinate to the cover $\{\hB_j\}$ with $\chi_j = 1$ on $B_j$, $j \geq 1$. We lift each of these functions from $\Si$ to the blowup of $\Si \ti \RP$. 
Finally, set $G_{j}:=U_{n_j}-\log \sigma_j$, where $G_0:=U_0 -\log |\al|=-\log y-\log|\al|$. Now define 
\begin{equation}
\hu_0:=\sum_{j=0}^N \chi_j G_{j}. 
\end{equation}
We compute that $N( \hu_0) = f_0$, where $f_0$ is polyhomogeneous and is bounded at the original boundary $\psi = 0$ and has
leading term of order $R^{-1}$ at each of the `front' faces where $R=0$. 

Our goal is to iteratively solve away all of the terms in the polyhomogeneous expansion of $f_0$. This must be done separately
at the two types of boundary faces.  It turns out to be necessary to first solve away the series at $R=0$ and after that the
series at $\psi=0$. The reason for doing things in this order is that, as we now explain, the iterative problem that must
be solved at the $R=0$ front faces is global on each hemisphere, and the solution `spread' to the boundary of this hemisphere,
i.e., where $\psi = 0$. By contrast, the iterative problem at the original boundary is completely local in the $y$ directions
and may be done uniformly up to the corner where $R = \psi = 0$, so its solutions do not spread back to the front faces.

For simplicity, we assume that there is only one front face, and we begin by considering the model case $(\CC \ti \RP)_0$, 
on which the linearization of \eqref{scaleEBE} at $U_n$ can be written
\begin{equation}
L_n =  -\del_R^2 - \frac{2}{R}\del_R - \frac{1}{R^2}\Delta_{S^2_+}  +  2r^{2n} e^{2U_n} = -\del_R^2 - \frac{2}{R}\del_R + \frac{1}{R^2}
\left( -\Delta_{S^2_+}  +  T(\psi)\right),
\label{Lnmodel}
\end{equation}
where the potential equals
\[
T(\psi) = \frac{(n+1)^2}{\sin^2 \psi S_n(\psi)^2}.
\]
In general terms, $L_n$ is a relatively simple example of an `incomplete iterated edge operator', as explained in more detail in 
\cite{MazzeoWitten2017}, based on the earlier development of this class in \cite{ALMP1,ALMP2}. We need relatively little of 
this theory here and quote 
from \cite{MazzeoWitten2017} as needed.   In the present situation, we can regard $L_n$ as a conic operator over the cross-section $S^2_+$.  (It
is the fact that this link of the cone itself has a boundary which makes $L_n$ an `iterated' edge operator.) 

The crucial fact is that the operator 
\[
J = -\Delta_{S^2_+}  +  T(\psi),
\]
induced on this conic link has discrete spectrum. The proof of this is based on the observation that $T(\psi) \sim 1/\psi^2$ as $\psi \to 0$.
It can then be shown using standard arguments, cf.\ \cite{ALMP1,ALMP2}, that the domain of $J$ as an unbounded operator on $L^2(S^2_+)$ is 
compactly contained in $L^2$. This implies the discreteness of the spectrum.  Another proof which provides more accurate information uses 
that $J$ is itself an incomplete uniformly degenerate operator, as analyzed thoroughly in \cite{Mazzeo1991}. The main theorem in that paper produces 
a particular degenerate pseudodifferential operator $G$ which invers $J$ on $L^2$. It is also shown there that $G: L^2(S^2_+) \to \psi^2 H^2_0(S^2_+)$ 
(where $H^2_0$ is the scale-invariant Sobolev space associated to the vector fields $\psi \del_\psi$, $\psi \del_\theta$).  The compactness 
of $\psi^2 H^2_0(S^2_+) \hookrightarrow L^2(S^2_+)$ follows from the $L^2$ Arzela-Ascoli theorem.  There is an accompanying regularity 
theorem:  if $(J - \lambda)w = f$ where (for simplicity) $f$ is smooth and vanishes to all orders at $\psi = 0$ and $\lambda \in \RR$ (or
more generally can be any bounded polyhomogeneous function), then $w$ is polyhomogeneous with an expansion of the form
\[
w \sim  \sum w_{j\ell}(\theta) \psi^{\gamma_j}(\log \psi)^\ell, \ \ w_{j\ell} \in \mathcal C^\infty(S^1).
\]
As usual, there are only finitely many log terms for each exponent $\gamma_j$.  These exponents are the indicial roots of the operator $J$,
and a short calculation shows that these satisfy $2 = \gamma_0 < \gamma_1 < \ldots$.   Note that the lowest indiical root equals $2$,
so solutions all vanish to at least order $2$ at $\psi = 0$, which is in accord with our knowledge about the behavior of solutions
to the linearization of \eqref{scaleEBE} at the model Nahm pole solution $-\log y$. 

Denote the eigenfunctions and eigenvalues of $J$ by $\mu_i(\psi, \theta)$ and $\lambda_i$.  Since $T(\psi) > 0$, each $\lambda_i > 0$. 
The restriction of $L_n$ to the $i^{\mathrm{th}}$ eigenspace is now an ODE $L_{n,i} = -\del_R^2 - 2R^{-2} \del_R + R^{-2}\lambda_i$.  Seeking
solutions of the form $R^\delta \mu_i(\psi,\theta)$ leads to the corresponding indicial roots 
\[
\delta_i^\pm = -\frac12 \pm \frac12 \sqrt{1 + 4\lambda_i},
\]
which are the only possible formal rates of growth or decay of solutions to $L_n u = 0$ as $R \to 0$.  To satisfy the generalized Nahm pole condition,
we only consider exponents greater than $-1$, i.e., the sequence $0 < \delta_1^+ < \delta_2^+ < \ldots$.   We now conclude the following
\begin{lemma}
Suppose that $f \sim \sum f_{j\ell}(\psi,\theta) R^{\gamma_j} (\log R)^\ell$ is polyhomogeneous at the face $R=0$ on $(\CC \ti \RP)_0$,
where all $f_{j\ell}$ are polyhomogeneous with nonnegative coefficients at $\psi = 0$ on $S^2_+$. Then there exists a polyhomogeneous 
function $u$ such that $L_n u = f + h$, where $h$ is polyhomogeneous at $\psi = 0$ and vanishes to all orders as $R \to 0$. 
At  $R \to 0$, $u \sim \sum u_{j\ell} R^{\gamma_j'} (\log R)^\ell$; the exponents $\gamma_j'$ are all of the form $\gamma_i + 2$,f 
where $\gamma_j$ appears in the list of exponents in the expansion for $f$, or else $\delta_i^+ + \ell$, $\ell \in \mathbb N$. Each 
coefficient function $u_{j\ell}$, as well as the entire solution $u$ itself and the error term $h$, vanish like $\psi^2$ at the boundary $\psi = 0$. 
\end{lemma}

Using the same result, we may clearly generate a formal solution to our semilinear elliptic equation in exactly the same way. Therefore,
using this Lemma, we may now choose a function $\hu_1$ which is polyhomogeneous on $(\Si \ti \RP)_{p_1 \ldots p_N}$ and 
such that $N(\hu_0 + \hu_1) = f_1$, where $f_1$ vanishes to all orders at $R=0$ and is polyhomogeneous and vanishes like $\psi^2$ at $\psi = 0$.
The lowest exponent in the expansion for $\hu_1$ equals $\min\{1, \delta_0^+ > 0\}$. 

The final step in our construction of an approximate solution is to carry out an analogous procedure at the original boundary $y=0$ away
from the front faces. This can be done almost exactly above. In this case, \eqref{modpsi} can be thought of as an ODE in $y$ with `coefficients' 
which are operators acting in the $z$ variables, so we are effectively just solving a family of ODE's parametrized by z. This may be done uniformly
up to the corner $R = \psi = 0$.  We omit details since they are the same as before.  We obtain after this step a final correction term $\hu_2$ 
which is polyhomogeneous and vanishes to all orders at $R=0$, and which satisfies
\[
N(\hu_0 + \hu_1 + \hu_2) = f,
\]
where $f$ vanishes to all orders at all boundaries of $(\Si \ti \RP)_{p_1 \ldots p_N}$. 

The calculations above are useful not just for calculating formal solutions to the problem, but also for understanding the regularity of
actual solutions to the nonlinear equation $N(u) = 0$ which satisfy the generalized Nahm pole boundary conditions with knots.
The new ingredient that must be added is a parametrix $G$ for the linearization of $N$ at the approximate solution $u_0$.
This operator $G$ is a degenerate pseudodifferential operator for which there is very precise information known concerning the 
pointwise behavior of the Schwartz kernel. This is explained carefully in \cite{MazzeoWitten2013} for the simple Nahm pole case
and in \cite{MazzeoWitten2017} for the corresponding problem with knot singularities. We shall appeal to that discussion and
the arguments there and simply state the
\begin{proposition} 
Let $u$ be a solution to \eqref{scaleEBE} which is of the form $u = u_0 + v$ where $v$ is bounded as $y \to 0$ (in particular
as $\psi \to 0$ and $R \to 0$).  Then $u$ is polyhomogeneous at the two boundaries $\psi = 0$ and $R=0$ of the blowup
$(\Si \ti \RP)_{p_1, \ldots, p_N}$, and its expansion is fully captured by that of $u_0$. 
\end{proposition}
\end{subsection}

\begin{subsection}{Existence of solutions}
We now come to the construction of solutions to \eqref{scaleEBE} on the entire space $\Si \ti \RP$ which satisfy the
asymptotic $\SLR$ conditions as $y \to \infty$ and which also satisfy the generalized Nahm pole boundary conditions
with knot singularities at $y=0$.   We employ the barrier method. The main ingredients in the construction of
the barrier functions are the approximate solutions $u_0$ and $u_\infty$ obtained above.

We first consider this problem in the simpler case.
\begin{proposition}
If $E = K^{1/2} \oplus K^{-1/2}$ and $\vp = \begin{pmatrix} 0 & 1 \\ \beta & 0 \end{pmatrix}$, i.e. there are no knot singularities,
then there exists a solution $u$ to \eqref{scaleEBE} which is smooth for $y > 0$, asymptotic to $u_\infty$ as $y \to \infty$,
(and which satisfies the Nahm pole boundary condition at $y=0$). 
\end{proposition}
\proof
Choose a smooth nonnegative cutoff function $\tau(y)$ which equals $1$ for $y \leq 2$ and which vanishes for $\tau \geq 3$,
and define $\hu = \tau(y) u_0 + (1-\tau(y))u_\infty$. We consider the operator
\[
\wh{N}(v) = N(\hu + v) = - (\del_y^2 + \Delta_{g_0} ) v +  e^{2\hu} (e^{2v} - 1) + |\beta|^2 e^{-2\hu} (1 - e^{-2v}) + f,
\]
where $f = N(\hu)$ is smooth on $\Si \ti \overline{\mathbb R}^+$, vanishes to infinite order at $y = 0$ and vanishes identically for $y \geq 3$. 

We now find barrier functions for this equation.  Indeed, we compute that if $0 < \epsilon < 1$, then
\[
\wh{N}(A y^\epsilon) = A \epsilon(1-\epsilon) y^{\epsilon-2} + e^{2\hu}( e^{2Ay^\epsilon} - 1) + |\beta|^2 e^{-2\hu}(1 - e^{-2Ay^\epsilon}) + f.
\]
The second and third terms on the right are nonnegative because $A y^\epsilon > 0$, and we can certainly choose $A$ sufficiently
large so that the entire right hand side is positive for all $y > 0$.

We can improve this supersolution for $y$ large. Indeed, 
\[
\wh{N}( A' e^{-\epsilon y}) \geq -A' \epsilon^2 e^{-\epsilon y} + e^{2\hu}( 2A' e^{-\epsilon y}) + |\beta|^2e^{-2\hu}(1 - e^{-2A'e^{-\epsilon y}}) + f,
\]
and if $\epsilon$ is sufficiently small and $A'$ is sufficiently large, then the entire right hand side is positive, at least for $y \geq 1$, say.

We now define $v^+ = \min\{ A y^\epsilon, A' e^{-\epsilon y}\}$. The calculations above show that $v^+$ is a supersolution to the equation.
Essentially the same equations show that $v^- = \max\{ -A y^{\epsilon}, -A' e^{-\epsilon y}\}$ is a subsolution.   

We now invoke Propostion~\ref{KWE} to conclude that there exists a solution $v$  to $\wh{N}(v) = 0$, or equivalently, a solution
$u = \hu + v$ to $N(u) = 0$, which satisfies $|u + \log y| \leq A y^\epsilon$ as $y \to 0$ and $|u - u_\infty| \leq A' e^{-\epsilon y}$
as $y \to \infty$.    The regularity theorem for \eqref{scaleEBE} shows that this solution is polyhomogeneous at $y=0$, and hence
must have an expansion of the same type as $\hu$, and a similar but more standard argument can be used to produce a better
exponential rate of decay as $y \to \infty$. 
\qed
\medskip

\begin{proposition}
Let $E = L \oplus L^{-1}$ and $\vp = \begin{pmatrix} 0 & \alpha \\ \beta & 0 \end{pmatrix}$ be a stable Higgs pair, and let
$(p_j, n_j)$ be the `knot data' determined by $\alpha$.  Then there exists a solution $u$ t o \eqref{scaleEBE} of the
form $u = \hu + v$ where $v \to 0$ as $y \to 0$ and as $y \to \infty$. 
\end{proposition}
\proof  We proceed exactly as before, writing 
\[
\wh{N}(v) = N(\hu + v) =  - (\del_y^2 + \Delta_{g_0}) v + |\alpha|^2 e^{2\hu}( e^{2v} - 1) + |\beta|^2 e^{-2\hu} (1 - e^{-2v}) + f,
\]
with $f = N(\hu)$ vanishing to all orders as $y \to 0$ and identically for $y \geq 3$. The same barrier functions obviously 
work in the region $y \geq 3$, and also in the region near $y=0$ away from the knot singularities.  

To construct barriers near a knot $(p,0)$ of weight $n$, recall the explicit structure of $\hu$ near this point and expand
the nonlinear term $e^{2v} - 1$ one step further to write in some small neighborhood of the front face created by
blowing up this point
\[
\wh{N}(v) = (-\del_R^2 - \frac{2}{R}\del_R + \frac{1}{R^2}( -\Delta_{S^2_+} + \tilde{T})) v +  k e^{2U_n}( e^{2v} - 1 - 2v) + |\beta|^2e^{-2U_n}(1 - e^{-2v}) + f.
\]
Here $k$ is a strictly positive function which contains all the higher order terms in the expansion for $\hu$, and $\tilde{T}$ is a slight
perturbation of the term $T$ appearing in the linearization $L_n$.   Let $\mu_0$ denote the ground state eigenfunction for this operator
on $S^2_+$. The corresponding eigenvalue $\lambda_0'$ is a small perturbation of $\lambda_0$, which we showed earlier was strictly greater
than $0$.  Now compute
\[
\wh{N} ( A R^\epsilon \mu_0(\psi,\theta)) = (\lambda_0' - \epsilon(\epsilon+1)) A R^{\epsilon-2} \mu_0 + f + E,
\]
where $E$ is the sum of the two terms involving $e^{\pm 2U_n}$.  As before, since $v \geq 0$ implies $e^{2v} - 1 - 2v \geq 0$ and $1 - e^{-2v}$,
we have that $E \geq 0$, and if $\epsilon$ is sufficiently small, then
this first term on the right has positive coefficient, and dominates $f$.   We have thus produced a local supersolution near $(p,0)$. 
The full supersolution is
\[
v^+ = \min \{ A R^{\epsilon}\mu_0,  A' y^{\epsilon/2},  A'' e^{-\epsilon y} \}.
\]
We have chosen to use the exponents $\epsilon$ and $\epsilon/2$ in the first two terms here in order to ensure that the first term
is smaller than the second in the interior of the front face $R=0$; indeed, $A R^\epsilon \mu_0< A' (R \sin \psi)^{\epsilon/2}$ when
$R < (A'/A)^{2/\epsilon} (\sin \psi)^{\epsilon/2}$.  This means that $v^+$ agrees with $A' y^{\epsilon/2}$ near the original boundary
and with $A R^\epsilon \mu_0$ near the other boundaries, and as before, with the exponentially decreasing term when $y$ is large. 

A very similar calculation with the same functions produces a subsolution $v^-$. Altogether, we deduce, by Proposition \ref{KWE} again, 
the existence of a solution $u = \hu + v$ to $N(u) = 0$ with the correct asymptotics.
\qed
\end{subsection}
\end{section}

\begin{section}{Uniqueness} 
In this section, we prove a uniqueness theorem for solutions of the \EBE satisfying the (generalized) Nahm pole boundary condition. This will be phrased
in terms of the associated Hermitian metrics. The key to this is the subharmonicity of the Donaldson metric, which we recall in the first subsection.

\begin{subsection}{The Distance on Hermitian metrics}
Suppose that $H$ is a Hermitian metric on a bundle $E$, with compatible data $(A,\phi,\phi_1)$, which satisfies the extended Bogomolny equations.
As we have discussed, it is possible to choose a holomorphic gauge which is parallel in the $y$ direction such that 
$\MD_1=\Bp$, $\MD_2= \mathrm{ad}\,\vp$, $\MD_3=\pa_y$. In this gauge, the Hermitian metric $H$ determines the gauge fields by
\begin{equation}
\pa^A=\pa+H^{-1}\pa H,\;\vp^{\st}=H^{-1}\vp^{\da}H,\;\pa^{\MA}_y=\pa^{A_y}+i\phi_1=\pa_y+H^{-1}\pa_yH,
\end{equation}
where of course $\pa$ is the complex differential on $\Si$ and in this trivialization $\vp^{\da} = \vp^{\da}=\bar{\vp}^{\top}$.
We can then write the \EBE as 
\[
\Bp(H^{-1}\pa H)+[\vp^{\st H},\vp]+h_0^2\pa_y(H^{-1}\pa_yH)=0,
\]
where $h_0^2|dz|^2$ is the Riemannian metric on $\Si$. 

Following \cite{donaldson1985anti}, we define the distance between Hermitian metrics
\begin{equation}
\sigma(H_1,H_2)=\Tr (H_1^{-1}H_2)+\Tr (H_2^{-1}H_1)-4,
\end{equation}
and recall from that paper two important properties:
\begin{itemize}
\item[1)] $\sigma(H_1,H_2)\geq 0$, with equality if and only if $H_1=H_2$;
\item[2)] A sequence of Hermitian metric $H_i$ converges to $H$ in the usual $\MC^0$ norm if and only if $\sup_\Si \sigma(H_i,H)\to 0$.
\end{itemize}

\begin{lemma}
Suppose that $H_1$ and $H_2$ are both harmonic metrics. Then the complex gauge transform $h:=H_1^{-1}H_2$ satisfies 
\begin{equation}
\Bp(h^{-1}\pa^{A_1}h)+\pa_y(h^{-1}\pa^{A_1}_yh)+[h^{-1}[\vp^{\st},h],\vp]=0.
\end{equation}
\end{lemma}
\proof In holomorphic gauge, 
$$
A_2=H_2^{-1}\pa H_2=h^{-1}H_1^{-1}\pa H_1 h+h^{-1}\pa h=H_1^{-1}\pa H_1+h^{-1}\pa^{A_1}h,
$$
hence $\Bp(H_2^{-1}\pa H_2)-\Bp(H_1^{-1}\pa H_1)=\Bp(h^{-1}\pa^{A_1}h)$. 

Similarly, 
$$
H_2^{-1}\pa_y H_2=H_1^{-1}\pa_y H_1+h^{-1}(\pa_y h+[H_1^{-1}\pa_y H_1,h])=H_1^{-1}\pa_y H_1+h^{-1}\pa_y^{\MA_y}h.
$$
Hence $\pa_y(H_2^{-1}\pa_y H_2)-\pa_y(H_1^{-1}\pa_y H_1)=\pa_y(h^{-1}\pa_y^{\MA_y}h).$

Finally,
\[
[\vp^{\st H_2},\vp]-[\vp^{\st H_1},\vp]=[h^{-1}[\vp^{\st H_1},h],\vp].
\]

Altogether, we deduce the stated equation from the harmonic metric equations
$$
\Bp(H_j^{-1}\pa H_j)+[\vp^{\st H_j},\vp]+h_0^2\pa_y(H_j^{-1}\pa_yH_j)=0,\ j = 1, 2. 
$$
\qed

We next show that $\sigma$ is subharmonic. 
\begin{proposition}
Define $h=H_1^{-1}H_2$ as above, where $H_1$ and $H_2$ satisfy the Extended Bogonomy equation. Then $(\Delta+\pa_y^2)\sigma\geq 0$ 
on $\Sigma\ti (0,+\infty)$.
\end{proposition}
\proof 
We first compute 
\begin{equation}
\begin{split}
\Bp\pa_z\Tr(h)&=\Tr(\Bp \pa^{A_1}h)\\
&=\Tr(\Bp(h h^{-1} \pa^{A_1}h))\\
&=\Tr(\Bp(h) h^{-1} \pa^{A_1}h)+\Tr(h\Bp(h^{-1} \pa^{A_1}h))\\
&\geq \Tr(h\Bp(h^{-1} \pa^{A_1}h)),
\end{split}
\end{equation}
since $\Tr(BhB^{\st})\geq 0$ for any matrix $B$.

Continuing on, 
\begin{equation}
\begin{split}
\pa_y^2\Tr(h)&=\Tr(\pa_y\pa_y^{\MA_1}h)\\
&=\Tr( (\pa_y h)h^{-1}\pa_y^{\MA_1}h)+\Tr(h(\pa_y(h^{-1}\pa_y^{\MA_1}h)))\\
&\geq \Tr(h(\pa_y(h^{-1}\pa_y^{\MA_1}h))),
\end{split}
\end{equation}
where we use $\pa_y=(\pa_y^{\MA_1})^{\st}$ and that $\st$ is the conjugate transpose with respect to $H_1$.

Finally, 
\begin{equation}
\begin{split}
0&=\Tr([[\vp^{\st},h],\vp])\\
&=\Tr([h,\vp]h^{-1}[\vp^{\st},h])+\Tr(h[h^{-1}[\vp^{\st},h],\vp]).
\end{split}
\end{equation}
Since $\Tr([h,\vp]h^{-1}[\vp^{\st},h])\geq 0$, we obtain $\Tr(h[h^{-1}[\vp^{\st},h],\vp])\leq 0$.

Putting these together gives 
\begin{equation}
\begin{split}
(\Bp\pa_z+h_0^2\pa_y^2)\Tr(h)&\geq \Tr(h\Bp(h^{-1} \pa^{A_1}h)+h_0^2h(\pa_y(h^{-1}\pa_y^{\MA_1}h)))\\
&\geq \Tr(h\Bp(h^{-1} \pa^{A_1}h)+h_0^2h(\pa_y(h^{-1}\pa_y^{\MA_1}h+h[h^{-1}[\vp^{\st},h],\vp])))\\
&\geq 0,
\end{split}
\end{equation}
and dividing by $h_0^2$ proves the claim. 
\qed
\end{subsection}

\begin{subsection}{Asymptotics of the Hermitian metric}
In order to apply the subharmonicity of $\sigma(H_1, H_2)$ from the last subsection, we need to understand the asymptotics of this function
near $y=0$. This, in turn, relies on a detailed examination of the asymptotics of the Hermitian metric.
\begin{proposition}
\label{HermitianmetricUniqueness}
Fix a Higgs pair $(E\cong L^{-1}\oplus L,\vp=\begin{pmatrix}
t & \al \\
\be & -t
\end{pmatrix})$.  For any $p \in\Si$, choose an open set $U_p$ around  $(p,0)$ in $\Si \ti \RP$.  Let $H$ be a solution to the Hermitian 
\EBE \eqref{algebraEBE}; as explained earlier, $H$ is polyhomogeneous on $(\Si \ti \RP)_{p_1 \ldots p_N}$ (where the $p_j$ are the zeroes of $\alpha$).

(1) Suppose in some local trivilization in $U_p$ that $\vp|_{U_x}=\begin{pmatrix}
0 & 1 \\
q & 0
\end{pmatrix}$, where $q$ is holomorphic.  Suppose also that 
\begin{equation}
H=\begin{pmatrix}
\MO(y^{-1}) & \MO(1) \\
\MO(1) & \MO(1)
\end{pmatrix}.
\end{equation}
Here $\MO(y^s)$ indicates a polyhomogeneous expansion with lowest order term a smooth multiple of $y^s$. Suppose also that 
$H$ satisfies the Nahm pole boundary condition in unitary gauge.  Then
\begin{equation}
H \sim \begin{pmatrix}
y^{-1}g_0+\MO(1)  & o(1) \\
o(1) & y g_0^{-1}+\MO(1)
\end{pmatrix},
\end{equation}
where $o(1)$ indicates a polyhomogeneous expansion with positive leading exponent.

(2) Suppose that in a local trivilization, $\vp|_{U_p}=\begin{pmatrix}
t & z^n \\
q & -t
\end{pmatrix}$ where $z=0$ is the point $p$ and $q$ holomorphic. If, in spherical coordinates
\begin{equation}
H=\begin{pmatrix}
\MO(y^{-1}R^{-n}) & \MO(1) \\
\MO(1) & \MO(1)
\end{pmatrix},
\end{equation}
then
\begin{equation}
H=\begin{pmatrix}
\MO(y^{-1}R^{-n}) & \MO(1) \\
\MO(1) & \MO(yR^n)
\end{pmatrix}
\label{Hknot}
\end{equation}
\end{proposition}
\proof
We first address (1). Write 
$H=\begin{pmatrix}
h_{11} & h_{12} \\
h_{21} & h_{22}
\end{pmatrix}$ and consider a gauge transformation $g$ for which $H=g^2$. Then $g^{\da}=g$ and 
$
g=\begin{pmatrix}
a & b \\
\bar{b} & d
\end{pmatrix}$ where $a$ and $d$ are real functions and $ad-b\bb=1$.  We then compute
\begin{equation}
\phi_z=g\varphi g^{-1}= 
\begin{pmatrix}
a & b \\
\bar{b} & d
\end{pmatrix} \, 
\begin{pmatrix}
0 & 1 \\
q & 0
\end{pmatrix} \, \begin{pmatrix}
d & -b \\
-\bar{b} & a
\end{pmatrix}
=\begin{pmatrix}
bdq-a\bb & -b^2q+a^2 \\
d^2q-\bb^2 & -bdq+a\bb
\end{pmatrix}.
\end{equation}
By proposition \ref{Nahmpoleexpansion}, the Nahm pole boundary condition requires that
\begin{equation}
bdq-a\bb\sim o(1),\ d^2q-\bb^2\sim o(1),\ -b^2q+a^2\sim \frac{g_0}{y} + \MO(1).
\label{expansionuniquenessnahm}
\end{equation}
By definition, $H=g^2=\begin{pmatrix}
a^2+b\bb & ab+bd \\
\bb a+\bb d & d^2+b\bb
\end{pmatrix}$. 
The leading terms of $d^2+b\bar{b}$ is positive, hence $b$ and $d$ are bounded. Combining this with \eqref{expansionuniquenessnahm} and 
the relation $ad-b\bar{b}=1$, we obtain 
\begin{equation}
a\sim y^{-\frac{1}{2}}g_0^{\frac{1}{2}},\;d\sim y^{\frac{1}{2}}g_0^{-\frac 12},\;b =  o(y^{\frac{1}{2}})
\end{equation}
and  thus 
\begin{equation}
H=\begin{pmatrix}
a^2+b\bb & ab+bd \\
\bb a+\bb d & d^2+b\bb
\end{pmatrix}=\begin{pmatrix}
y^{-1}g_0+ o(y^{-1}) & o(1) \\
o(1) & y g_0^{-1}+ o(y)
\end{pmatrix}.
\end{equation}

As for (2), we compute 
\begin{equation}
\begin{split}
\phi_z=g\varphi g^{-1}  & =  
\begin{pmatrix}
a & b \\
\bar{b} & d
\end{pmatrix}\, 
\begin{pmatrix}
t & z^n \\
q & -t
\end{pmatrix}\,
\begin{pmatrix}
d & -b \\
-\bar{b} & a
\end{pmatrix} \\ 
& = \begin{pmatrix}
bdq-a\bb z^n+atd+|b|^2t & -b^2q+a^2z^n-2bat \\
d^2q-z^n\bb^2+2td\bar{b} & -bdq+a\bb z^n-|b|^2t-adt
\end{pmatrix}.
\end{split}
\end{equation}
By Proposition \ref{Knotsingularityexpansion}, the knot singularity implies that 
\begin{equation}
bdq-a\bb z^n+atd+|b|^2t \sim \MO(1),\ -b^2q+a^2z^n-2ba\sim z^{n}e^{U_n}+\cdots,\ 
-bdq+a\bb z^n-|b|^2t-adt \sim \MO(1).
\end{equation}
As before, $H=g^2=\begin{pmatrix}     a^2+b\bb & ab+bd \\     \bb a+\bb d & d^2+b\bb \end{pmatrix}$ where $d^2+b\bb\sim \MO(1)$, so by the 
same positivity, $d$ and $b$ are both $\MO(1)$.  Next, $e^{U_n}=f(\psi)/yR^n$ where $f$ is regular.  From $-b^2q+a^2z^n-2ba\sim z^{n}e^{U_n}$ 
we get $a\sim y^{-\frac{1}{2}}R^{-\frac{n}{2}}$. In addition, since $ab+bd = \MO(1)$ and $ad-b\bb=1$, we see that $b\sim y^{\frac12}R^{\frac{n}{2}}$, 
so $d\sim y^{\frac 12}R^{\frac n2}$. Altogether, $H$ has the form \eqref{Hknot}.
\qed

\begin{proposition}
Suppose $H_j=\begin{pmatrix} p_j & q_j \\  q_j^{\da} & s_j \end{pmatrix}$, $j = 1, 2$, are two solutions which both satisfy the Nahm pole boundary 
condition at $y = 0$ and have the same limit as $y\to \infty$. Then $H_1=H_2$.
\end{proposition}
\proof
By Propositions \ref{metricaprior} and \ref{HermitianmetricUniqueness}, we see that as $y\to 0$, $p_j\sim y^{-1}g_0+\cdots$, $s_j\sim yg_0^{-1}+\cdots$, 
$q_j\sim o(1)$.   We claim that this implies that $\sigma(H_1,H_2)  \to 0$ as $y \to 0$.  First, 
\[
H_1^{-1}H_2=\begin{pmatrix} s_1p_2-q_1q_2^{\da} & \star \\ \star & -q_1^{\da}q_2+p_1s_2 \end{pmatrix},
\]
so 
\begin{equation}
\Tr(H_1^{-1}H_2)=s_1p_2-q_1q_2^{\da}-q_1^{\da}q_2+p_1s_2 = 2+o(1).
\end{equation}
The same holds for $\Tr(H_2^{-1}H_1)$. This proves the claim. 

We have now see that $\sigma(H_1, H_2)$ is nonnegative and subharmonic, and approaches $0$ as $y \to 0$ and also as $y \to \infty$,
hence $\sigma(H_1, H_2) \equiv 0$, i.e., $H_1 = H_2$. 
\qed

\begin{proposition}
Let $H_1$ and $H_2$ be two Hermitian metrics which are both solutions with a knot singularity of degree $n$ at $(p,0)$. Then there
exists a constant $C$ such that $\sigma(H_1, H_2)\leq C$ in a neighborhood $U$ of $(p,0)$. 
\end{proposition}
\proof
Write $H_j=\begin{pmatrix} a_j & b_j \\  b_j^{\da} & d_j \end{pmatrix}$, $j = 1, 2$. By Propositions \ref{metricaprior} 
and \ref{HermitianmetricUniqueness}, 
$$
a_j\sim y^{-1}R^{-n},\ d_j \sim yR^{n}, \ b_j = o(1),\ b_j^{\da}  = o(1).
$$
Thus $\Tr(H_1H_2^{-1})=a_1d_2-b_1b_2^{\da}-b_1^{\da}b_2+d_1a_2 = \MO(1)$, and similarly, $\Tr(H_2^{-1}H_1) = \MO(1)$. The result 
follows immediately.  
\qed

\medskip

We next recall the Poisson kernel of $\Delta_g = \Delta_{g_0} + \pa_y^2$. For any $p \in \Si$, $P_p(z,y)$ is the unique function on $\Si \ti \RP$ 
which satisfies $\Delta_{g} P_q(z,y) = 0$, $P|_{y=0} =\delta_q$, and $P(z,y)\to 1/\mbox{Area}(\Si)$ as $y\to \infty$.

\begin{theorem}
Suppose that there exist two Hermitian metrics $H_1$, $H_2$ which are solutions and satisfy the Nahm pole boundary condition with 
knot singularities at $p_j$ of degree $n_j$, as determined by the component $\alpha$ in the Higgs field 
$\vp=\begin{pmatrix} t & \al \\ \be & -t \end{pmatrix}$. Suppose also that $H_1$ and $H_2$ have the same limit as 
$y \to \infty$. Then $H_1 = H_2$. 
\end{theorem}
\proof
By Proposition \ref{HermitianmetricUniqueness}, $\sigma(H_1,H_2) \to 0$ as $y \to 0$ and $z \notin \{p_1, \ldots, p_N\}$. Near each $p_j$
there is a neighbourhood $U_j$ where $\sigma(H_1,H_2)|_{U_j}\leq C$.  

Now define $Q(z,y)$ to equal the sum of Poisson kernels $\sum_{j=1}^N P_{p_j}(z,y)$. Then for any $\epsilon > 0$, $(\Delta_{g_0}+\del_y^2)
(\sigma(H_1, H_2) -\epsilon Q)\geq0$, and $\sigma(H_1, H_2)-\epsilon Q\leq 0$ as $y\to0$ and as $y\to\infty$. This means that $\sigma(H_1, H_2) \leq 
\epsilon Q$.  Since this is true for every $\epsilon > 0$, we conclude that $\sigma(H_1, H_2) \leq 0$, i.e., $H_1 = H_2$. 
\qed
\end{subsection}
\end{section}

\begin{section}{Solutions with Knot Singularities on $\CC \ti \RP$}
We now consider the \EBE on $\CC \ti \RP$ with generalized Nahm pole boundary conditions and a finite number of knot singularities.

\begin{subsection}{Degenerate Limit}
Consider a trivial bundle $E$ over $\CC\ti\RP$, as in \cite{witten2011fivebranes} and \cite{gaiotto2012knot}, the limiting behavior of the classical Jones polynomial indicates that one expects
that for solutions of the \EBE on $\CC \ti \RP$, $\phi\to 0$ and $\phi_1\to 0$ as $y\to \infty$. The equation $\MD_3\vp=0$ also implies that 
the conjugacy class of $\vp$ is independent of $y$, and as argued in these papers, this implies if $Q$ is any invariant polynomial, then 
$\pa_yQ(\vp)=0$, hence that $\vp$ is necessarily nilpotent. 

Based on these heuristic considerations, we consider a trivial rank $2$ holomorphic bundle over $\mathbb{C}$ and assume 
$\vp=\begin{pmatrix} 
0 & p(z) \\
0 & 0
\end{pmatrix}$. We can assume $p(z)$ is a polynomial as up to a complex gauge transform the equivalent class of the Higgs bundle only depends on the zeros of the upper trangular part of $\vp$. In general, the vanishing section determined by the line bundle has the form
$s=\begin{pmatrix}  R(z)  \\  S(z)  \end{pmatrix}$.  Consider the section $K(z):=( s\wedge\vp(s))(z)= p(z)S(z)^2$ of the determinant bundle,
which we can naturally identify with a holomorphic function on $\CC$. Its zero set defines a positive divisor $D$. 

If the singular monopoles all have order 1, as $K(z):=( s\wedge\vp(s))(z)= p(z)S(z)^2$, we obtain that $S(z)$ will not have zeros. Up to a complex gauge transform $g=\begin{pmatrix}
1 & -\frac RS\\
0 & 1
\end{pmatrix}$, we can assume in the same trivialization, $\vp=\begin{pmatrix} 0 & p(z) \\ 0 & 0 \end{pmatrix}$ and $s=\begin{pmatrix}  0  \\ 1 \end{pmatrix}$. In general, we can only assume $\vp=\begin{pmatrix} t & p \\ q & -t \end{pmatrix}$ and the vanishing line bundle correspond to $s=\begin{pmatrix}  0  \\ 1 \end{pmatrix}$ with the nilpotent condition that $t^2+pq=0$.

Although we expect to be able to solve \EBE with knot singularities corresponding to any divisor, the equation will generally not reduce
to a scalar one, except in the special case where $\vp=\begin{pmatrix}
0 & p(z) \\
0 & 0
\end{pmatrix}$ and $s=\begin{pmatrix}  0  \\  1 \end{pmatrix}$ and it gives an $SL(2,\mathbb{R})$ structure. Now the \EBE reduce to 
\begin{equation}
-(\Delta+\pa_y^2)v+|p(z)|^2e^{2v}=0,
\label{HEBEEuc}
\end{equation}
and we shall search for a solution for which $v\to -C\log y$ as $y\to \infty$.

\begin{remark}
It is not enough to simply require that $v \to -\infty$ as $y\to \infty$. Indeed, if $p(z) \equiv 1$, then $z$-independent solutions solve the ODE 
$-u''+e^{2u}=0$. One solution is $-\log y$, but there is an additional family $\log (\frac{C}{\sinh (Cy)})$ for any $C > 0$. These are the only 
global solutions to this ODE. 
The solutions in this second family grow like $-Cy$ as $y\to\infty$, and $\phi_1\to C\begin{pmatrix}  \frac{i}{2} & 0 \\  0 & -\frac{i}{2} \end{pmatrix}$.
These solutions appear in \cite{kronheimer1990instantons} and is described by Gaiotto and Witten \cite{gaiotto2012knot} as a real symmetry breaking 
phenomenon at $y\to\infty$.
\end{remark}
\end{subsection}

\begin{subsection}{Existence}
In this section, we will prove the
\begin{proposition}
Let $p(z)$ be any polynomial on $\CC$ of degree $N_0 > 1$. Then there exists a solution $u$ to \eqref{HEBEEuc} satisfying the generalized
Nahm pole conditions with knot determined by the divisor $D = \sum n_j p_j$ of the polynomial $p$, and which is asymptotic
to $-(N_0 + 1) \log R - \log \sin \psi + \MO(1)$ as $R \to \infty$, uniformly in $(\psi, \theta) \in S^2_+$. 
\end{proposition}
\proof 
As before, first construct a function $\hu$ which is an approximate solution to this equation with
boundary conditions to all order in all asymptotic regimes, and then use the method of barriers to find a correction
term which gives the exact solution. 

We first pass to the blowup of $\CC \ti \RP$ around the points $(p_j, 0)$, and in an additional step, also take the radial
compactification as $R \to \infty$.  This gives a compact manifold with corners which we call $\wh{X}$ for simplicity; 
there are boundary faces $F_1, \ldots, F_N$, each hemispheres corresponding to the blowups at the zeros of $p$, another
boundary face $F_\infty$, also a hemisphere, corresponding to the radial compactification at infinity, and the original
boundary $B$, which is a disk with $N$ smaller disks removed.  

The first step in the construction of $\hu$ is to use the approximate solutions near each of these faces. 
Around $F_j$, $j = 1, \ldots, N$, we use $U_{n_j}$;  near $F_\infty$ we use $U_{N_0}$, but now of course with
$R\to \infty$ rather than near $0$, and finally near $B$ we ue $-\log y$. Pasting these together gives a polyhogeneous 
function $\hu_0$ on $\wh X$ for which $N(\hu_0) = f_0$ blows up like $1/R_j$ near each $F_j$, decays like
$R^{-3}$ near $F_\infty$, and blows up like $-\log y$ near $y=0$. Here we are denoting the nonlinear operator
by $N$ as before. 

The second step is to correct the expansions, or equivalently, to solve away the terms in the expansions of $f_0$,
at each of these boundary faces.  Near each $F_j$ this is done exactly as in the last section.  Near $F_\infty$
it is done in a completely analogous manner, solving away the terms of order $R^{-3-j}$ using correction terms
of order $R^{-1-j}$.  Near $F_j$ we are using the solvability of the operator $J_{n_j}$, while near $F_\infty$ we
use the operator $J_{N_0}$.  Finally, exactly as before, we solve away the terms in the expansion of the remainder
as $y \to 0$ along $B$. This may be done uniformly up to the boundaries of $B$.  Taking Borel sums of each
of these expansions, there exists a polyhomogeneous function $\hu_1$ on $\wh X$ which satisfies
$N( \hu_0 + \hu_1) = f_1$ where $f_1$ vanishes to all orders at every boundary component of $\wh X$. 
The approximate solution is $\hu = \hu_0 + \hu_1$. 

Now write $\wh N (v) = N( \hu + v)$. We expand this as
\[
\wh N(v) = -\Delta_{\bar g} v + e^{2\hu} |p(z)|^2 (e^{2v} - 1)+f_1.
\]

We construct a supersolution using the following three constituent functions:  first, $ R^{-\epsilon}_\infty \mu_0^{N_0}$ near $F_\infty$ (where
$\mu_0^{N_0}$ is the ground state eigenfunction for $J_{N_0}$); next, $R_j^\epsilon \mu_0^{n_j}$ near $F_j$. Finally, $y^{\epsilon/2}$ near $B$. 
We then take
\[
v^+ = \min \{ R_\infty^{-\epsilon} \mu_0^{N_0}, R_1^\epsilon \mu_0^{n_1}, \ldots, R_N^\epsilon \mu_0^{n_N}, y^{\epsilon/2}\}. 
\]
It is straightforward to check that $\wh N(v^+) \geq 0$.     With the obvious changes,
we also obtain a function $v^-$ for which $\wh N(v^-) \leq 0$.  

Proposition \ref{KWE} now implies that there exists a solution $v$ to this equation. By construction, $u = \hu + v$ satisfies
all the required boundary conditions.
\qed

As in Section 4, this existence theorem is accompanied by some sharp estimates for the solution $u$. 
\begin{proposition}
The solution $u$ obtained in the previous proposition is polyhomogeneous on $\wh X$. In particular, it has a full asymptotic
expansion as $R \to \infty$, where the leading term is the model solution $U_{N_0}$. 
\end{proposition}

This, in turn, leads to a uniqueness theorem for the scalar equations:
\begin{theorem}
Let $p(z)$ be a polynomial on $\mathbb{C}$ of degree $N_0>1$. Suppose that $u_1$ and $u_2$ are two solutions to \eqref{HEBEEuc} 
satisfying the generalized Nahm pole conditions with knot determined by the zeroes of polynomial $p$ at $y=0$. Assume also that
as $R\to\infty$, $u_i \sim U_{N_0}+R^{-\epsilon}$, $i=1,2$. Then $u_1=u_2$.
\end{theorem}
\proof
By \eqref{HEBEEuc}, 
\[
-(\Delta+\pa_y^2)(u_1-u_2)+|p(z)|^2 (e^{2u_1} - e^{2u_2}) = \left(-(\Delta+\pa_y^2)+|p(z)|^2 F(u_1, u_2)\right)w = 0
\]
Here $w = u_1 - u_2$ and $F(u_1, u_2) = (e^{2u_1} - e^{2u_2})/(u_1 - u_2)$. By assumption that both $u_1$ and $u_2$ satisfy
the same boundary conditions, and using the regularity theory for solutions, we obtain that $\lim_{y \to 0} w = 0$, while
by the hypothesis on decay at infinity, $\lim_{R \to \infty} w = 0$ as well. Noting that $F(u_1, u_2) \geq 0$, no matter whether
$u_1 < u_2$ or $u_1 \geq u_2$, the maximum principle implies that $w \equiv 0$, i.e., $u_1 \equiv u_2$. 
\qed

\end{subsection}
\end{section}

\medskip

\bibliographystyle{plain}
\bibliography{references}
\end{document}